\documentclass[a4paper,mathserif,11pt]{article}

\usepackage{a4wide}
\usepackage[T1]{fontenc}
\usepackage[latin1]{inputenc}
\usepackage[english]{babel}
\usepackage{authblk}
\usepackage{amsmath,amsfonts,amsthm,amssymb}
\usepackage{cite}
\usepackage{stmaryrd}
\usepackage{bbold}
\usepackage{paralist}
\usepackage[normalem]{ulem}
\usepackage[bold]{hhtensor}
\usepackage{mathabx}
\usepackage{nicefrac}
\usepackage{tikz}
\usetikzlibrary{arrows}
\usepackage{pgfplots,pgfplotstable}
\pgfqkeys{/pgfplots}{
  cycle list name = black white
}
\usepackage{hyperref}
\usepackage{verbatimbox}
\usepackage{setspace}

\newtheorem{remark}{Remark}
\newtheorem{lemma}[remark]{Lemma}
\newtheorem{proposition}[remark]{Proposition}

\newcommand{\email}[1]{\href{mailto:#1}{#1}}

\newcommand{\N}{\mathbb{N}}
\newcommand{\Z}{\mathbb{Z}}
\newcommand{\R}{\mathbb{R}}
\newcommand{\eps}{\varepsilon}
\newcommand{\dis}{\displaystyle}

\renewcommand{\vec}[1]{\boldsymbol{#1}}
\newcommand{\grad}{\vec{\nabla}}
\renewcommand{\div}{{\rm div}}

\newcommand{\norm}[2][]{{\left\lVert#2\right\rVert}_{#1}}

\newcommand{\Ltwo}[1][{\cal D}]{{L^2(#1)}}
\newcommand{\Ltwon}[1][{\cal D}]{{L^2_{\rm n}(#1)}}

\title{On the best constant matrix approximating an oscillatory matrix-valued coefficient in divergence-form operators}
\author[1,3]{Claude Le~Bris\footnote{\email{claude.le-bris@enpc.fr}}}
\affil[1]{\'Ecole des Ponts ParisTech, CERMICS, 6 et 8 avenue Blaise Pascal, 77455 Marne-la-Vall\'ee Cedex 2, France}
\author[2,3]{Fr\'ed\'eric Legoll\footnote{Corresponding author:~\email{frederic.legoll@enpc.fr}}}
\affil[2]{\'Ecole des Ponts ParisTech, Laboratoire Navier - UMR 8205, 6 et 8 avenue Blaise Pascal, 77455 Marne-la-Vall\'ee Cedex 2, France}
\author[1,3]{Simon Lemaire\footnote{\email{simon.lemaire@epfl.ch}}}
\affil[3]{Inria Paris, MATHERIALS project-team, 2 rue Simone Iff, CS 42112, 75589 Paris Cedex 12, France}

\begin{document}

\maketitle

\begin{abstract}
We approximate an elliptic problem with oscillatory coefficients using a problem of the same type, but with constant coefficients. We deliberately take an engineering perspective, where the information on the oscillatory coefficients in the equation can be incomplete. A theoretical foundation of the approach in the limit of infinitely small oscillations of the coefficients is provided, using the classical theory of homogenization. We present a comprehensive study of the implementation aspects of our method, and a set of numerical tests and comparisons that show the potential practical interest of the approach. The approach detailed in this article improves on an earlier version briefly presented in~\cite{LBLeL:13}.
\end{abstract}

\section{Introduction} \label{se:intro}

\subsection{Context}

Consider the simple, linear, elliptic equation
\begin{equation} \label{eq:osc.0}
-\div(A_\eps\grad u_\eps) = f \ \ \text{in ${\cal D}$}, \qquad u_\eps = 0 \ \ \text{on $\partial {\cal D}$},
\end{equation}
in divergence-form, where ${\cal D}\subset\R^d$, $d\geq 1$, is an open, bounded domain which delimits what we hereafter call 'the physical medium', and where $A_\eps$ is a possibly random oscillatory matrix-valued coefficient. We suppose that all the requirements are satisfied so that problem~\eqref{eq:osc.0} is well-posed. In particular, we assume that $A_\eps$ is bounded and bounded away from zero uniformly in $\eps$. Our assumptions will be detailed in Section~\ref{sse:ass} below. The subscript $\eps$ encodes the characteristic scale of variation of the matrix field $A_\eps$. For instance, one may think of the case $A_\eps(\vec{x})=A^{\rm per}(\vec{x}/\eps)$ for a fixed $\Z^d$-periodic matrix field $A^{\rm per}$, although all what follows is not restricted to that particular case.

It is well-known that, for $\eps$ small (comparatively to the size of ${\cal D}$), and not necessarily infinitesimally small, the direct computation of the solution to~\eqref{eq:osc.0} is expensive since, in order to capture the oscillatory behavior of $A_\eps$ and $u_\eps$, one has to discretize the domain ${\cal D}$ with a meshsize $h\ll\eps$. The computation becomes prohibitively expensive in a multi-query context where the solution $u_\eps(f)$ is needed for a large number of right-hand sides $f$ (think,~e.g., of a time-dependent model where~\eqref{eq:osc.0}, or a similar equation, should be solved at each time step~$t^n$ with a right-hand side $f(t^n)$, or of an optimization loop with $f$ as an unknown variable, where~\eqref{eq:osc.0} would encode a distributed constraint). Alternatives to the direct computation of~$u_\eps$ exist. Depending on the value of $\eps$, the situation is schematically as follows.

\begin{itemize}
\item[$\bullet$] For $\eps<\overline{\eps}$, where $\overline{\eps}$ is a given, medium-dependent threshold (typically $\overline{\eps}\approx{\rm size}({\cal D})/10$), one can consider that homogenization theory~\cite{BeLiP:78,JiKoO:94,Tarta:10} provides a suitable framework to address problem~\eqref{eq:osc.0}. That theory ensures the existence of a limit problem for infinitely small oscillations of the coefficient $A_\eps$. The limit problem reads
\begin{equation} \label{eq:hom.0}
-\div(A_\star\grad u_\star) = f \ \ \text{in ${\cal D}$}, \qquad u_\star = 0 \ \ \text{on $\partial {\cal D}$}. 
\end{equation}
The matrix-valued coefficient $A_\star$ is (i) non-oscillatory, (ii) independent of $f$, and (iii) given by an abstract definition that can become more or less explicit, depending on the assumptions concerning the structure of $A_\eps$ (and the probabilistic setting in the random case). The solution to the homogenized problem~\eqref{eq:hom.0} can be considered an accurate $L^2$-approximation of the oscillatory solution to~\eqref{eq:osc.0} as soon as the size $\eps$ of the oscillations of $A_\eps$ is sufficiently small.

There are several cases for which the abstract definition giving $A_\star$ can be made explicit. The simplest examples are (i) periodic coefficients of the form $A_\eps(\vec{x})=A^{\rm per}(\vec{x}/\eps)$, with $A^{\rm per}$ a $\Z^d$-periodic matrix field, and (ii) stationary ergodic coefficients of the form $A_\eps(\vec{x},\omega)=A^{\rm sto}(\vec{x}/\eps,\omega)$, with $A^{\rm sto}$ a (continuous or discrete) stationary matrix field. In both cases, one can prove that $A_\star$ is a deterministic constant (i.e.~independent of $\vec{x}$) matrix, for which a simple explicit expression is available. Whenever a corrector (in the terminology of homogenization theory, see~\cite{BeLiP:78,JiKoO:94,Tarta:10} and~\eqref{eq:corr.sta}--\eqref{eq:corr.per} below) exists, it is in addition possible to reconstruct an $H^1$-approximation of the solution to~\eqref{eq:osc.0}, using the solutions to the corrector problem and to the homogenized problem~\eqref{eq:hom.0}.

Practically, whenever an explicit definition is available for $A_\star$, one can compute an approximation of the oscillatory solution to~\eqref{eq:osc.0} by solving the non-oscillatory problem~\eqref{eq:hom.0}. The advantage is obviously that the latter can be solved on a coarse mesh. The cost of the method then lies in the offline computation of $A_\star$.

\item[$\bullet$] For $\eps\geq\overline{\eps}$, the size of the oscillations is too large to consider that homogenization theory provides a suitable framework to approximate problem~\eqref{eq:osc.0}, and one may use, in order to efficiently compute an approximation of $u_\eps$, dedicated numerical approaches.

Classical examples include the Variational Multiscale Method (VMM) introduced by Hughes et~al.~\cite{HuFeM:98}, and the Multiscale Finite Element Method (MsFEM) introduced by Hou and Wu~\cite{HouWu:97} (see also the textbook~\cite{EfHou:09}). We also refer to the more recent works by M\r{a}lqvist and Peterseim~\cite{MalPe:14} (on the Local Orthogonal Decomposition (LOD) method), or Kornhuber and Yserentant~\cite{KorYs:15}, on localization and subspace decomposition. Many more examples of approaches are available in the literature.

The MsFEM approach (as well as the LOD approach) is essentially based on an offline/online decomposition of the computations. In the first step, local problems are solved at the microscale, in order to compute oscillatory basis functions. Each basis function is obtained by solving an oscillatory problem posed on a macro-element or on a patch of macro-elements. These oscillatory problems do not depend on the right-hand side $f$, and are independent one from another. In the second step, the global problem, which depends on the right-hand side $f$, is solved. The second step is performed,~e.g., by considering a Galerkin approximation on the multiscale discrete space built in the offline step. The original online cost of solving an oscillatory problem on a fine mesh (using a discrete space at one single fine scale) is reduced to solving an oscillatory problem on a coarse mesh consisting of macro-elements (using a multiscale discrete space).

These methods provide an $H^1$-approximation of the oscillatory solution $u_\eps$. Note that they are (a priori) applicable without any restriction on the structure of $A_\eps$, and are also applicable, and indeed applied, in the regime $\eps<\overline{\eps}$. Note also that, in the stochastic setting, the computations must be performed $\omega$ by $\omega$, for ``each'' realization $\omega$ of the random environment.

The finite element Heterogeneous Multiscale Method (HMM) introduced by E~and~Engquist~\cite{EEngq:03} 
is another popular multiscale technique. It is however based on a different perspective. Its aim is to compute an approximation of the coarse solution $u_\star$ by means of local averages of the oscillatory coefficient $A_\eps$. 
\end{itemize}

One way or another, all these approaches rely on the knowledge of the coefficient $A_\eps$. It turns out that there are several contexts where such a knowledge is incomplete, or sometimes merely unavailable. From an engineering perspective (think,~e.g., of experiments in Mechanics), there are numerous prototypical situations where the response $u_\eps(f)$ can be measured for {\em some} loadings $f$, but where $A_\eps$ is {\em not} completely known. In these situations, it is thus not possible to use homogenization theory, nor to proceed with any MsFEM-type approach or with the similar approaches mentioned above.

\medskip

We have discussed above two possibilities to address multiscale problems such as~\eqref{eq:osc.0}, using either the homogenization theory or dedicated numerical approaches. Restricting our discussion to homogenization theory, we can identify three limitations, quite different in nature, to the practical application of the theory:
\begin{itemize}
\item[$\bullet$] First, homogenization theory has been developed in order to address the case of infinitely small oscillations of the coefficients, and is hence not appropriate for media such that $\eps\geq\overline{\eps}$. In practice, one may for instance want to evaluate the effective coefficients (such as the Poisson ratio and the Young modulus for problems in Mechanics) of a medium for which $\eps\geq\overline{\eps}$. It is always possible (if an explicit definition is available) to compute $A_\star$, considering on purpose the (fictitious) limit of infinitely small oscillations, but there is no reason for that $A_\star$ to be an accurate approximation of the medium it is supposed to describe. 
\item[$\bullet$] Assume that an explicit expression is available for $A_\star$. A practical limitation is that, in most cases except for the somewhat ideal case of periodic coefficients (with a known period), the computation of $A_\star$ by classical methods is expensive. For instance, in the stochastic setting, the computation of $A_\star$ requires to solve, many times, a corrector problem set on a truncated approximation of an asymptotically infinitely large domain. This is especially challenging in the stationary ergodic case with long-range correlations. Note that equivalent limitations appear for MsFEM-type or similar approaches in the stochastic setting.
\item[$\bullet$] Another evident limitation shows up when one examines the homogenized limit of~\eqref{eq:osc.0} for a coefficient $A_\eps$ such that no explicit expression is available for $A_\star$ (although $A_\eps$ is well-known, and although the homogenized limit of~\eqref{eq:osc.0} is known to read as~\eqref{eq:hom.0}). This case might occur as soon as $A_\eps$ is not the rescaling $A(\cdot/\eps)$ of a simple (periodic, quasi-periodic, random stationary, \dots) function $A$.
\end{itemize}

Finding a pathway alternate to standard approaches is thus a practically relevant question. Given our discussion above, we are interested in approaches valid for the different regimes of~$\eps$, which make no use of the knowledge on the coefficient $A_\eps$, but only use some (measurable) responses of the medium (obtained for certain given solicitations). Questions similar in spirit, but different in practice, have been addressed two decades ago by Durlofsky in~\cite{Durlo:91}. They are similar in spirit because the point is to define an effective coefficient only using outputs of the system. They are however different in practice because the effective matrix is defined by upscaling, and hence the approach of~\cite{Durlo:91} is local. This approach is indeed based on considering, in a representative elementary volume, some particular problems (with zero loading and suitable boundary conditions), for which the solutions in the case of homogeneous coefficients are affine and write as independent of these homogeneous coefficients. Considering $d$ choices of such problems (that is, $d$ choices of boundary conditions), and postulating the equality of the fluxes respectively resulting from the original oscillatory and homogeneous equivalent problems, one determines the coefficients of an ``effective'' matrix. Several variants exist in the literature, as well as many other approaches.

The original approach we introduce in this article improves on an earlier version briefly presented in~\cite{LBLeL:13}. Our approach is global, in the sense that it uses the responses of the system in the {\em whole} domain ${\cal D}$. Note of course that it can be used locally as an upscaling technique, for instance in problems featuring a prohibitively large number of degrees of freedom. 

\medskip

In passing, we note that our approach provides, at least in some settings, a characterization of the homogenized matrix which is an alternative to the standard characterization of homogenization theory (see Proposition~\ref{pr:as.cons} below). To the best of our knowledge, this characterization has never been made explicit in the literature. 

\medskip

Throughout this article, we restrict ourselves to cases when problem~\eqref{eq:osc.0} admits (possibly up to some extraction) a homogenized limit that reads as problem~\eqref{eq:hom.0}, where the homogenized matrix coefficient
\begin{equation*} 
A_\star\text{ is {\em deterministic} and {\em constant}.}
\end{equation*}
This restrictive assumption on the class of $A_\star$ (and thus on the structure of the coefficient $A_\eps$ in~\eqref{eq:osc.0}, and on the probabilistic setting in the random case) is useful for our theoretical justifications, but not mandatory for the approach to be applicable (see Section~\ref{sec:outline} below).

\subsection{Presentation of our approach}
\label{sec:pres}

We now sketch, for a coefficient $A_\eps$ that we take for simplicity deterministic, the idea underlying our approach. Let ${\cal S}$ denote the set of real-valued $d\times d$ positive-definite symmetric matrices.

For any constant matrix $\overline{A} \in {\cal S}$, consider generically the problem with {\em constant} coefficients
\begin{equation} \label{eq:bar.0}
-\div(\overline{A}\grad\overline{u})=f \ \ \text{in ${\cal D}$}, \qquad \overline{u} = 0 \ \ \text{on $\partial {\cal D}$}.
\end{equation}
We investigate, for any value of the parameter $\eps$, how we may define a constant matrix $\overline{A}_\eps \in {\cal S}$ such that the solution $\overline{u}_\eps$ to problem~\eqref{eq:bar.0} with matrix $\overline{A} = \overline{A}_\eps$ best approximates the solution $u_\eps$ to~\eqref{eq:osc.0}. Note that, since $\overline{A}_\eps$ is constant, its skew-symmetric part plays no role in~\eqref{eq:bar.0}. We hence cannot hope for characterizing the skew-symmetric part of $\overline{A}_\eps$. Without loss of generality, we henceforth make the additional assumption that the homogenized matrix $A_\star$ is {\em symmetric} and that we seek a best (constant) symmetric matrix. Should $A_\star$ not be symmetric, it is replaced in the sequel by its symmetric part. In~\cite{LBLeL:13}, the constant matrix $\overline{A}_\eps$ is defined as a minimizer of
\begin{equation} \label{eq:infsup.old}
\inf_{\overline{A} \in {\cal S}} \; \sup_{f\in L^2({\cal D}), \, \norm[\Ltwo]{f}=1} \ \norm[L^2({\cal D})]{u_\eps(f)-\overline{u}(f)}^2,
\end{equation}
where we have emphasized the dependency upon the right-hand side $f$ of the solutions to~\eqref{eq:osc.0} and~\eqref{eq:bar.0}. The use of a $L^2$ norm in~\eqref{eq:infsup.old} (and not of e.g. a $H^1$ norm) is reminiscent of the fact that, for sufficiently small $\eps$, we wish the best constant matrix to be close to $A_\star$, and that $u_\eps$ converges to $u_\star$ in the $L^2$ norm but not in the $H^1$ norm.

Note that problem~\eqref{eq:infsup.old} is only based on the knowledge of the outputs $u_\eps(f)$ (that could be, e.g., experimentally measured), and {\em not} on that of $A_\eps$ itself. Note also that, in practice, we cannot maximize upon all right-hand sides $f$ in $L^2({\cal D})$ (with unit norm). We therefore have to replace the supremum in~\eqref{eq:infsup.old} by a maximization upon a finite-dimensional set of right-hand sides, which we will have to select thoughtfully (see Section~\ref{ssse:sup}). 

In this article, we keep the same type of characterization for $\overline{A}_\eps$ as in~\cite{LBLeL:13} (that is, through an inf-sup problem), but we use a slightly different cost function than in~\eqref{eq:infsup.old}. The constant matrix $\overline{A}_\eps$ is here defined as a minimizer of
\begin{equation} \label{eq:infsup.0}
\inf_{\overline{A} \in {\cal S}} \ \sup_{f\in L^2({\cal D}), \, \norm[\Ltwo]{f}=1} \ \norm[L^2({\cal D})]{(-\Delta)^{-1}\left(\div(\overline{A}\grad u_\eps(f))+f\right)}^2,
\end{equation}
where $(-\Delta)^{-1}$ is the inverse laplacian operator supplied with homogeneous Dirichlet boundary conditions: for any $g\in H^{-1}({\cal D})$, $z = (-\Delta)^{-1}g$ is the unique solution in $H^1_0({\cal D})$ to
\begin{equation*} 
-\Delta z=g \ \ \text{in ${\cal D}$}, \qquad z=0 \ \ \text{on $\partial{\cal D}$}.
\end{equation*}

The cost function of~\eqref{eq:infsup.0} is related to the one of~\eqref{eq:infsup.old} through the application, inside the $L^2$ norm of the latter, of the zero-order differential operator $(-\Delta)^{-1}\left(\div(\overline{A}\grad \cdot)\right)$. Note that, in sharp contrast with~\eqref{eq:infsup.old}, the function $\dis \norm[L^2({\cal D})]{(-\Delta)^{-1}\left(\div(\overline{A}\grad u_\eps(f))+f\right)}^2$ used in~\eqref{eq:infsup.0} is a polynomial function of degree 2 in terms of $\overline{A}$, a property which brings stability and significantly speeds up the computations. The specific choice~\eqref{eq:infsup.0} has been suggested to us by Albert Cohen (Universit\'e Pierre et Marie Curie). 

\begin{remark}
The reason to choose $f \in L^2({\cal D})$ in~\eqref{eq:infsup.0}, rather than $f \in H^{-1}({\cal D})$, is discussed in Remark~\ref{rem:L2} below.
\end{remark}

Several criteria can be considered to assess the quality and the usefulness of our approach:
\begin{enumerate}[(i)]
\item \label{it:as.cons} {\em asymptotic consistency:} does the sequence $\left\{ \overline{A}_\eps \right\}_{\eps>0}$ of best matrices, defined as minimizers of~\eqref{eq:infsup.0}, converge, when $\eps$ goes to $0$, to the homogenized matrix $A_\star$? If this is indeed the case, the approach provides an approximation for the homogenized matrix alternate to standard homogenization (note, in particular, that our approach does not require solving a corrector problem).
\item \label{it:eff} {\em efficiency:} practically, is this best matrix $\overline{A}_\eps$ efficiently computable? In particular, how many right-hand sides does its computation really require?
\item \label{it:L.two} {\em $L^2$-approximation:} for any fixed $\eps$, not necessarily small, how well does the solution $\overline{u}_\eps$ to~\eqref{eq:bar.0} with matrix $\overline{A}_\eps$ approximate the reference solution $u_\eps$ to~\eqref{eq:osc.0} in the $L^2$ norm?
\item \label{it:H.one} {\em $H^1$-approximation:} using $\overline{A}_\eps$, is it possible to reconstruct (if possible for a marginal additional cost) an accurate approximation of $u_\eps$ in the $H^1$ norm? Recall that in homogenization theory, a corrector problem must be solved to compute the homogenized matrix, but once this is performed, one can reconstruct an $H^1$-approximation of $u_\eps$ using the solution of the latter problem at no additional cost.
\end{enumerate}

\subsection{Outline and perspectives}
\label{sec:outline}

The article is organized as follows. To begin with, we introduce in Section~\ref{se:prel} the assumptions we will make throughout the article, and we recall the basics of homogenization. We formalize our approach in Section~\ref{se:forma}. We establish an asymptotic consistency result (thereby positively answering to Question~\eqref{it:as.cons} above, see Proposition~\ref{pr:as.cons}), and we explain how the best matrix we compute can be used to construct an approximation in the $H^1$ norm of the oscillatory solution (hence addressing Question~\eqref{it:H.one} above). We also detail how to approximate the infinite-dimensional space $\dis \left\{f\in L^2({\cal D}), \ \ \norm[\Ltwo]{f}=1 \right\}$ present in~\eqref{eq:infsup.0} by a finite-dimensional space of the form $\dis \text{Span} \, \left\{ f_p, \ 1 \leq p \leq P \right\}$ for some appropriate functions $f_p$ (see~\eqref{eq:VPn} below). In Section~\ref{se:impl}, we explain how the problem of finding the best constant matrix can be efficiently solved in practice (thereby answering to Question~\eqref{it:eff}). 

Finally, in Section~\ref{se:num}, we present, as a practical answer to Questions~\eqref{it:as.cons},~\eqref{it:eff},~\eqref{it:L.two} and~\eqref{it:H.one}, a number of representative numerical experiments, both in the periodic and stationary ergodic settings, and we provide some comparison with the classical homogenization approach. We show in particular that choosing a small number $P$ of right-hand sides (in practice, we often set $P=d(d+1)/2$) is sufficient for our approach to provide accurate results. 

\bigskip

We emphasize that the aim of the numerical experiments described in Section~\ref{se:num} is different in the periodic setting and in the stochastic setting. In the former case, computing the homogenized matrix is inexpensive, and thus we cannot hope for our approach (which requires solving highly oscillatory equations) to outperform the classical homogenization approach in terms of efficiency. The periodic setting is hence to be considered as a validation setting. 

The situation is entirely different in the stochastic setting, which is much more challenging. In that setting, our approach can compete as far as Questions~\eqref{it:eff},~\eqref{it:L.two} and~\eqref{it:H.one} are concerned. We show that, for an essentially identical computational cost compared to the standard homogenization approach, our approach allows us to compute a more accurate approximation of the solution $u_\eps$ to the highly oscillatory equation, both in $L^2$ and in $H^1$ norms.

More importantly, the reader should bear in mind that our approach targets practical situations where the information on the oscillatory coefficients in the equation may be incomplete. The comparison with standard homogenization approaches which is performed in Section~\ref{se:num} is hence somewhat unfair for our approach, as the former approaches need a complete knowledge of the coefficient $A_\eps$, whereas ours does not.

\bigskip

There are several possible follow-ups for this work:
\begin{itemize}
\item First, one can perform a detailed study of the robustness of the approach with respect to imprecise data, assuming for instance that we only have access locally to {\em coarse} averages of the outputs $u_\eps(f)$ or $\grad u_\eps(f)$.
\item Second, the extension to nonlinear equations may be studied, where the oscillatory problem is formulated as the optimization problem
$$
\inf\left\{\int_{\cal D}K\left(\frac{\vec{x}}{\eps},\grad u(\vec{x})\right){\rm d}\vec{x}-\int_{\cal D}f(\vec{x})u(\vec{x})\,{\rm d}\vec{x}, \quad u \in W^{1,p}_0({\cal D})\right\},
$$
where the function $\vec{\xi} \in \R^d \mapsto K(\cdot,\vec{\xi})$ is strictly convex. In a multi-query context, our approach (and this is also true for other approaches) is even more interesting for nonlinear equations than for linear ones. Indeed, however large the parameter $\eps$ is, solving a nonlinear oscillatory equation for a large number of right-hand sides is prohibitively expensive. In contrast, in the linear case, as soon as the LU decomposition of the stiffness matrix can be computed and stored,~i.e.~as soon as $\eps$ is not too small, the cost for computing several solutions becomes almost equal to the cost for computing one. The computational workload thus remains affordable. This is not the case in a nonlinear context.

\item Third, the approach may be extended to homogenized matrices that are not constant. Indeed, as soon as some additional information is available on $A_\star$, one could adequately modify the search space for $\overline{A}$ in~\eqref{eq:infsup.old} or~\eqref{eq:infsup.0}. For instance, the case of a slowly varying matrix $A_\star(\vec{x})$, depending upon $\vec{x} \in {\cal D}$ in a sense to be made precise, can be considered. Following a suggestion by Albert Cohen, it may also be possible to balance the dimension of the space in which $\overline{A}$ is searched with the amount of noise present in the problem (which is related to the value of $\eps$) and the number of fine-scale solutions that are available (here the dimension $P$ of the space~\eqref{eq:VPn} introduced below).
\end{itemize}

\section{Preliminaries} \label{se:prel}

We describe the stationary ergodic setting we adopt. This setting includes, as a particular case, the periodic case. For a more detailed presentation of the particular stochastic setting we here consider, we refer to the theoretically-oriented articles~\cite{BLBLi:06,BLBLi:07}, to the numerically-oriented articles~\cite{BouPi:04,LeBri:10}, and to the review article~\cite{ACLBL:11} (as well as to the extensive bibliography contained therein). For more insight on stochastic homogenization in general, we refer the reader to the seminal contribution~\cite{PapaV:81}, to~\cite{EnSou:08} for a numerically-oriented presentation, as well as to the classical textbooks~\cite{BeLiP:78,JiKoO:94}. The reader familiar with that theory may easily skip this section and directly proceed to Section~\ref{se:forma}.

\subsection{Assumptions} \label{sse:ass}

Recall that ${\cal D}$ denotes an open, bounded subset of $\R^d$, $d\geq 1$. Let $(\Omega,{\cal Z},\mathbb{P})$ be a probability space, on which we assume an ergodic structure, and let $\displaystyle \mathbb{E}(X) = \int_\Omega X(\omega)\,{\rm d}\mathbb{P}(\omega)$ be the expectation of any random variable $X\in L^1(\Omega,{\rm d}\mathbb{P})$. We consider problem~\eqref{eq:osc.0}, which reads, in the stochastic setting, as
\begin{equation} \label{eq:osc}
-\div(A_\eps(\cdot,\omega)\grad u_\eps(\cdot,\omega))=f \ \ \text{a.s.~in ${\cal D}$}, \qquad u_\eps(\cdot,\omega)=0 \ \ \text{a.s.~on $\partial{\cal D}$},
\end{equation}
where the function $f \in \Ltwo$ is independent of $\eps$ and deterministic (see Remark~\ref{rem:L2} below for a discussion on the choice of taking $f$ in $\Ltwo$).

We assume that
\begin{equation} \label{eq:ass.sta}
A_\eps(\vec{x},\omega)=A^{\rm sto}(\vec{x}/\eps,\omega),
\end{equation}
where $A^{\rm sto}$ is such that there exist deterministic real numbers $\alpha,\beta>0$ such that
\begin{equation} \label{eq:ass.bound.sta}
A^{\rm sto}(\cdot,\omega)\in L^\infty(\R^d;{\cal S}_{\alpha,\beta})\quad\text{almost surely},
\end{equation}
with
\begin{equation*}
{\cal S}_{\alpha,\beta} = \left\{ M\in \R^{d \times d}, \ \ \text{$M$ is symmetric}, \ \ \alpha \, |\vec{\xi}|^2 \leq \vec{\xi}^T M \vec{\xi} \leq \beta \, |\vec{\xi}|^2 \ \text{for any $\vec{\xi}\in\R^d$} \right\}.
\end{equation*}
In addition, we assume that $A^{\rm sto}$ is a discrete stationary matrix field. A complete description of the discrete stationary ergodic setting we here consider can be found,~e.g.,~in the review article~\cite[Section 2.2]{ACLBL:11}. For brevity, we only mention here that the purpose of this setting is to formalize the fact that, even though realizations may vary, the matrix $A^{\rm sto}$ at point $\vec{y} \in \R^d$ and the matrix $A^{\rm sto}$ at point $\vec{y}+\vec{k}$, $\vec{k} \in \Z^d$, share the same probability law. The local, microscopic environment (encoded in the oscillatory matrix field $A_\eps(\vec{x},\omega)=A^{\rm sto}(\vec{x}/\eps,\omega)$) has a $\eps \mathbb{Z}^d$-periodic structure {\em on average}.

Assumption~\eqref{eq:ass.bound.sta} ensures the existence and uniqueness of the solution to~\eqref{eq:osc} in $H^1_0({\cal D})$, almost surely. Furthermore, almost surely, the solution $u_\eps(\cdot,\omega)$ to~\eqref{eq:osc} converges (strongly in $L^2({\cal D})$ and weakly in $H^1({\cal D})$) to some $u_\star \in H^1_0({\cal D})$ solution to~\eqref{eq:hom.0}, where the homogenized matrix $A_\star$ is {\em deterministic}, {\em constant} and belongs to ${\cal S}_{\alpha,\beta}$. As is well-known, $A_\star$ is independent of the right-hand side $f$ in~\eqref{eq:osc}. 

\begin{remark}
The above discussion is not restricted to the {\em discrete} stationary setting. We could as well have considered the {\em continuous} stationary setting, where the probability law of $A(\vec{y},\omega)$ does not depend on $\vec{y}$.
\end{remark}

\begin{remark}
  \label{re:darcy}
The form of the homogenized equation~\eqref{eq:hom.0} is in this context identical to that of the original equation~\eqref{eq:osc.0}. This is not a general fact. Although definite conclusions are yet to be obtained, there are all reasons to believe that the practical approach we introduce in this article carries over to cases where the homogenized equation is of a different form.
\end{remark}

The periodic setting is a particular case of the above discrete stationary setting, when $A$ is independent of $\omega$. This amounts to assuming that 
\begin{equation} \label{eq:ass.per}
A_\eps(\vec{x})=A^{\rm per}(\vec{x}/\eps),
\end{equation}
with $A^{\rm per}$ a $\Z^d$-periodic matrix field such that
\begin{equation} \label{eq:ass.bound.per}
A^{\rm per}\in L^\infty(\R^d;{\cal S}_{\alpha,\beta}).
\end{equation}

\subsection{Classical homogenization approach} \label{sse:comp.hom}

We briefly recall here the basics of homogenization. We focus the presentation on the stationary ergodic setting. The easy adaptation to the periodic setting is briefly commented upon. 

\medskip

Let $Q = (0,1)^d$. In the discrete stationary ergodic setting, the (deterministic, constant and symmetric) homogenized matrix $A_\star$ reads, for all $1\leq i,j\leq d$, as
\begin{equation} \label{eq:Astar}
\left[ A_\star \right]_{i,j} = \mathbb{E} \left( \int_Q \left( \vec{e}_i + \grad w_{\vec{e}_i}(\vec{y},\cdot) \right)^T \, A^{\rm sto}(\vec{y},\cdot) \, \left( \vec{e}_j + \grad w_{\vec{e}_j}(\vec{y},\cdot) \right) \, {\rm d}\vec{y} \right),
\end{equation}
where $(\vec{e}_1,\ldots,\vec{e}_d)$ denotes the canonical basis of $\R^d$, and where, for any $\vec{p}\in\R^d$, $w_{\vec{p}}$ is the solution (unique up to the addition of a random constant) to the so-called {\em corrector equation}
\begin{equation} \label{eq:corr.sta}
\left\{
\begin{array}{l}
-\div\left(A^{\rm sto}(\cdot,\omega)(\vec{p} + \grad w_{\vec{p}}(\cdot,\omega))\right)=0\quad\text{a.s.~in~$\R^d$},
\\
\grad w_{\vec{p}}\text{ is stationary},\qquad \displaystyle\mathbb{E}\left(\int_Q\grad w_{\vec{p}}(\vec{y},\cdot)\,{\rm d}\vec{y}\right)=0.
\end{array}
\right.
\end{equation}
In the periodic case $A_\eps(\vec{x})=A^{\rm per}(\vec{x}/\eps)$, the corrector equation reads as
\begin{equation} \label{eq:corr.per}
\left\{
\begin{array}{l}
-\div\left(A^{\rm per}(\vec{p} + \grad w_{\vec{p}})\right)=0\quad\text{in~$\R^d$},
\\
w_{\vec{p}}\text{ is $\Z^d$-periodic},
\end{array}
\right.
\end{equation}
and the homogenized matrix $A_\star$ is given by
\begin{equation*} 
\left[A_\star\right]_{i,j} = \int_Q \left( \vec{e}_i + \grad w_{\vec{e}_i}(\vec{y}) \right)^T \, A^{\rm per}(\vec{y}) \, \left( \vec{e}_j + \grad w_{\vec{e}_j}(\vec{y}) \right) \, {\rm d}\vec{y}.
\end{equation*}

In sharp contrast with the periodic case where, precisely by periodicity, it is sufficient to solve the corrector equation~\eqref{eq:corr.per} on the unit cell $Q$, the corrector equation~\eqref{eq:corr.sta} must be solved in the discrete stationary ergodic setting on the entire space $\R^d$. As pointed out in the introduction, this is computationally challenging. In practice, one often considers a truncated corrector equation posed, for an integer $N\neq 0$, on a large domain $Q^N = (-N,N)^d$:
\begin{equation} \label{eq:corr.N}
-\div\left(A^{\rm sto}(\cdot,\omega)(\vec{p} + \grad w^N_{\vec{p}}(\cdot,\omega))\right)=0 \ \ \text{a.s.~in $Q^N$}, \qquad w^N_{\vec{p}}(\cdot,\omega) \text{ is a.s.~$Q^N$-periodic.}
\end{equation}
The random matrix $A_\star^N(\omega)$, approximation of the deterministic homogenized matrix $A_\star$ given by~\eqref{eq:Astar}, is defined, for all $1\leq i,j\leq d$, by
\begin{equation} \label{eq:Astar.N}
\left[A_\star^N(\omega)\right]_{i,j} = \frac{1}{|Q^N|} \int_{Q^N} \left( \vec{e}_i + \grad w^N_{\vec{e}_i}(\vec{y},\omega) \right)^T \, A^{\rm sto}(\vec{y},\omega) \, \left( \vec{e}_j + \grad w^N_{\vec{e}_j}(\vec{y},\omega) \right) \, {\rm d}\vec{y}.
\end{equation}
Almost surely, it converges, in the limit of infinitely large domains $Q^N$, i.e. when $N\to +\infty$, to the (deterministic) matrix $A_\star$ (see~\cite{BouPi:04}). Since $A_\star^N(\omega)$ is random, it is natural to consider $M$ independent and identically distributed (i.i.d.) realizations of the field $A^{\rm sto}$, say $\left\{ A^{\rm sto}(\cdot,\omega_m) \right\}_{1\leq m\leq M}$, solve~\eqref{eq:corr.N} and compute~\eqref{eq:Astar.N} for each of them, and define
\begin{equation} \label{eq:Astar.N.M}
A_\star^{N,M} = \frac{1}{M}\sum_{m=1}^M A_\star^N(\omega_m)
\end{equation}
as a practical approximation to $A_\star$. Owing to the strong law of large numbers, we have that $\dis \lim_{N \to \infty} \lim_{M \to \infty} A_\star^{N,M} = A_\star$ almost surely.

\section{Formalization of our approach} \label{se:forma}

The approach we introduce below applies, up to minor changes, to both the periodic and the stationary ergodic settings. We however recall from Section~\ref{sec:outline} that only the stochastic setting (and more difficult cases) is practically relevant for our approach. For simplicity and clarity, we first present the full study of the approach in the periodic setting (see Sections~\ref{sse:infsup}, \ref{sse:as.cons} and~\ref{sse:hone}). We next discuss its extension to the stationary ergodic setting in Section~\ref{sse:stoch}.

\subsection{{\rm Inf}\,$\sup$ formulation} \label{sse:infsup}

As exposed in the introduction and expressed in formula~\eqref{eq:infsup.0}, we are going to seek a constant, symmetric, positive-definite matrix $\overline{A}_\eps$, so that problem~\eqref{eq:bar.0} with matrix $\overline{A}_\eps$ best approximates problem~\eqref{eq:osc.0}. To do so, we consider the problem introduced in~\eqref{eq:infsup.0}, that is
\begin{equation} \label{eq:infsup}
I_\eps = \inf_{\overline{A}\in{\cal S}} \ \sup_{f\in\Ltwon} \ \Phi_\eps(\overline{A},f),
\end{equation}
where $\Ltwon = \left\{f\in\Ltwo, \ \ \norm[\Ltwo]{f}=1\right\}$ and where, for any $\overline{A} \in \R^{d \times d}_{\rm sym}$ (the space of $d \times d$ real symmetric matrices) and any $f\in\Ltwo$,
\begin{equation} \label{eq:Phi}
\Phi_\eps(\overline{A},f) = \norm[\Ltwo]{ (-\Delta)^{-1} \left( \div ( \overline{A} \grad u_\eps(f) ) + f \right) }^2.
\end{equation}
Note that formula~\eqref{eq:Phi} is well-defined since $\div(\overline{A}\grad u_\eps(f))$ clearly belongs to $H^{-1}({\cal D})$ for all $\overline{A} \in \R^{d \times d}_{\rm sym}$ and $f\in\Ltwo$. We observe, as briefly mentioned in Section~\ref{sec:pres}, that the cost function $\Phi_\eps(\cdot,f)$ depends quadratically upon $\overline{A}$. From a computational viewpoint, in an iterative algorithm solving~\eqref{eq:infsup.0} or~\eqref{eq:infsup} that successively optimizes on $f$ and $\overline{A}$, minimizing $\Phi_\eps$ with respect to $\overline{A}$ for a fixed $f\in\Ltwon$ thus reduces to the simple inversion of a small linear system with $d(d+1)/2$ unknowns (see Section~\ref{ssse:inf.app}). This is in sharp contrast with our former formulation~\eqref{eq:infsup.old}. Of course, in both formulations~\eqref{eq:infsup.old} or~\eqref{eq:infsup.0}, for $\eps$ fixed, it is not guaranteed that our numerical algorithm captures the value $I_\eps$ defined by~\eqref{eq:infsup}. It only captures an approximation of it.

For both approaches~\eqref{eq:infsup.old} and~\eqref{eq:infsup.0}, one can prove an asymptotic consistency result for the sequence $\left\{ \overline{A}_\eps \right\}_{\eps>0}$: see Proposition~\ref{pr:as.cons} below in the case of~\eqref{eq:infsup.0} and~\cite{LBLeL:13} in the case of~\eqref{eq:infsup.old}. As the proof is essentially identical for both approaches, we only detail it for the present choice~\eqref{eq:infsup.0} (see Appendix~\ref{se:proof} below) and briefly point out to the case~\eqref{eq:infsup.old} considered in~\cite{LBLeL:13} in Remark~\ref{re:cras} below. 

\medskip

In order to gain further insight, and before stating the asymptotic consistency result, we first study, separately and for a fixed value of $\eps$, the maximization and minimization problems involved in~\eqref{eq:infsup}.

\subsubsection{The $\sup$ problem} \label{ssse:sup}

We show here that, for any fixed $\overline{A}\in{\cal S}$, the maximization problem over $f$ that is involved in~\eqref{eq:infsup}, namely $\dis \sup_{f\in\Ltwon} \ \Phi_\eps(\overline{A},f)$, is attained, and discuss how it can be solved in practice.

Let $\overline{A}\in{\cal S}$ be given. We introduce the notation
$$
\Delta_{\overline{A}} = \div(\overline{A}\grad\cdot),
$$
and let $(-\Delta_{\overline{A}})^{-1}$ be the operator defined by: for any $g\in H^{-1}({\cal D})$, $z = (-\Delta_{\overline{A}})^{-1}g$ is the unique solution in $H^1_0({\cal D})$ to
\begin{equation*} 
-\div(\overline{A}\grad z) = g \ \ \text{in ${\cal D}$}, \qquad z=0 \ \ \text{on $\partial{\cal D}$}.
\end{equation*}
We denote by ${\rm L}^{-1}_{\eps}$ the linear, compact and positive-definite operator from $\Ltwo$ to $\Ltwo$ such that, for any $f\in\Ltwo$, ${\rm L}^{-1}_{\eps}f = u_\eps(f)$, where $u_\eps(f)$ is the unique solution in $H^1_0({\cal D})$ to~\eqref{eq:osc.0}. Starting from~\eqref{eq:Phi}, it can be easily shown that
\begin{equation} \label{eq:Phi.op}
\Phi_\eps(\overline{A},f)=\int_{\cal D}{\cal H}_\eps^{\overline{A}}(f)\;f,
\end{equation}
where 
\begin{equation} \label{eq:Heps}
{\cal H}_\eps^{\overline{A}}(f) = \Big( \left( {\rm L}^{-1}_{\eps} \right)^\star \, \Delta_{\overline{A}}  \, (-\Delta)^{-1} + (-\Delta)^{-1} \Big) \Big( (-\Delta)^{-1}  \, \Delta_{\overline{A}} \, {\rm L}^{-1}_{\eps} + (-\Delta)^{-1} \Big) \, f
\end{equation}
is a compact, self-adjoint and positive semi-definite linear operator from $\Ltwo$ to $\Ltwo$. 
The eigenvalues of ${\cal H}_\eps^{\overline{A}}$ are thus nonnegative real numbers forming a sequence that converges to zero. We denote by $\lambda_{\eps,{\rm m}}^{\overline{A}}$ and $f_{\eps,{\rm m}}^{\overline{A}}$ the largest eigenvalue of ${\cal H}_\eps^{\overline{A}}$ and an associated normalized eigenvector, respectively. In view of~\eqref{eq:Phi.op}, we have
$$
\sup_{f\in\Ltwon}\Phi_\eps(\overline{A},f)=\lambda_{\eps,{\rm m}}^{\overline{A}}
$$
and the supremum is attained at $f_{\eps,{\rm m}}^{\overline{A}}$, which is hence a solution to the $\sup$ problem involved in~\eqref{eq:infsup}.

\medskip

In practice, instead of looking for the largest eigenvalue (and the associated eigenvector) of ${\cal H}_\eps^{\overline{A}}$ in the infinite-dimensional space $\Ltwon$, our approach consists in approximating this space $\Ltwon$ by a finite-dimensional subspace of the form
\begin{equation} \label{eq:VPn}
V^P_{\rm n}({\cal D}) = \left\{ f\in\Ltwon \ \text{s.t. there exists} \ \vec{c} = \{ c_p \}_{1\leq p\leq P} \in \R^P, \ \ |\vec{c}|^2=1, \ \ f=\sum_{p=1}^Pc_pf_p\right\},
\end{equation}
where $(f_1,\ldots,f_P)$ is an orthonormal family of functions in $\Ltwo$. 

\medskip

We discuss the choice of the dimension $P$ and of the family of functions $\{ f_p \}_{1\leq p\leq P}$. First of all, in the light of Lemma~\ref{le:recons} below (see also Section~\ref{sse:as.cons}), it seems in order to choose the dimension of $V^P_{\rm n}({\cal D})$ such that $\displaystyle P \geq d(d+1)/2$.

We now proceed, considering the regime $\eps$ small. Let $\overline{A} \neq A_\star$ be fixed. Homogenization theory states that, for $\eps$ sufficiently small, the operator ${\rm L}_\eps^{-1}$ (considered as an operator from $\Ltwo$ to $\Ltwo$) is close to the operator $(-\Delta_{A_\star})^{-1}$. Thus the operator ${\cal H}_\eps^{\overline{A}}$ defined by~\eqref{eq:Heps} is expected to be well-approximated by
\begin{equation} \label{eq:op.lim}
{\cal H}^{\overline{A}}_\star = 
\Big( (-\Delta_{A_\star})^{-1} \, \Delta_{\overline{A}} \, (-\Delta)^{-1} + (-\Delta)^{-1} \Big) \Big( (-\Delta)^{-1} \, \Delta_{\overline{A}} \, (-\Delta_{A_\star})^{-1} + (-\Delta)^{-1} \Big).
\end{equation}
Up to the extraction of a subsequence, the eigenvector $f^{\overline{A}}_{\eps,{\rm m}}$ we are seeking thus satisfies, by homogenization theory on eigenvalue problems,
$$
\lim_{\eps\to 0}\norm[\Ltwo]{f^{\overline{A}}_{\eps,{\rm m}}-f^{\overline{A}}_{\star,{\rm m}}}=0,
$$
where $f^{\overline{A}}_{\star,{\rm m}}$ is a normalized eigenvector associated with the largest eigenvalue of ${\cal H}^{\overline{A}}_\star$. In view of the expression~\eqref{eq:op.lim} of the limit operator, it seems natural to choose for the family of functions $\{ f_p \}_{1\leq p\leq P}$ the first $P$ (normalized) eigenvectors of the laplacian operator in the domain ${\cal D}$. For small values of $\eps$, say $\eps<\overline{\eps}$, we show that considering $\displaystyle P = d(d+1)/2$ functions $f_p$ is sufficient. This threshold $d(d+1)/2$ is at least intuitive thinking at the case of a constant symmetric matrix $\overline{A}$ and the set of equations $\dis \sum_{1 \leq i,j \leq d} -\overline{A}_{i,j} \ \partial_{ij} u_p = f_p$. In order to determine the $d(d+1)/2$ coefficients $\overline{A}_{i,j}$, the correct number of right-hand sides $f_p$ to consider is $d(d+1)/2$. The fact that it is indeed sufficient is made precise in the proof of Proposition~\ref{pr:as.cons} below (see in particular Lemma~\ref{le:recons}) and in Remark~\ref{re:infmax} below.

When the parameter $\eps$ takes larger values, say $\eps\geq\overline{\eps}$, the operator ${\cal H}^{\overline{A}}_\eps$ cannot be anymore approximated by the operator~\eqref{eq:op.lim} (with constant coefficients), and it may thus be necessary in that case to consider a larger number $\displaystyle P > d(d+1)/2$ of functions. We refer to Section~\ref{se:num} for concrete examples.

\begin{remark}
\label{rem:L2}
We discuss here why we have chosen to work with right-hand sides $f$ of the equation (e.g.~\eqref{eq:osc}) in $\Ltwo$ rather than in $H^{-1}({\cal D})$. We have here considered $\dis \sup_{f \in L^2({\cal D})} \ \frac{\Phi_\eps(\overline{A},f)}{\| f \|^2_{L^2({\cal D})}}$, and we could have considered $\dis \sup_{f \in H^{-1}({\cal D})} \ \frac{\Phi_\eps(\overline{A},f)}{\| f \|^2_{H^{-1}({\cal D})}}$.

Since $\Ltwo \subset H^{-1}({\cal D})$, we of course have $\dis \sup_{f \in H^{-1}({\cal D})} \ \frac{\Phi_\eps(\overline{A},f)}{\| f \|^2_{H^{-1}({\cal D})}} \geq \sup_{f \in L^2({\cal D})} \ \frac{\Phi_\eps(\overline{A},f)}{\| f \|^2_{H^{-1}({\cal D})}}$. Using the density of $\Ltwo$ in $H^{-1}({\cal D})$ and the continuity of $\Phi_\eps(\overline{A},\cdot)$ in $H^{-1}({\cal D})$, we actually get 
\begin{equation}
\label{eq:ehoh}
\sup_{f \in H^{-1}({\cal D})} \frac{\Phi_\eps(\overline{A},f)}{\| f \|^2_{H^{-1}({\cal D})}} = \sup_{f \in L^2({\cal D})} \frac{\Phi_\eps(\overline{A},f)}{\| f \|^2_{H^{-1}({\cal D})}}.
\end{equation}
The right-hand side of~\eqref{eq:ehoh} is of course different from the quantity $\dis \sup_{f \in L^2({\cal D})} \frac{\Phi_\eps(\overline{A},f)}{\| f \|^2_{L^2({\cal D})}}$, which we have considered in this article. Our choice is motivated by the fact that it is easier {\em in practice} to manipulate functions of unit $L^2$-norm. From the {\em theoretical} viewpoint, similar results would have been obtained with the left-hand side of~\eqref{eq:ehoh}.
\end{remark}

\subsubsection{The $\inf$ problem} \label{ssse:inf}

We discuss here how to efficiently solve the minimization problem over $\overline{A}$ that is involved in~\eqref{eq:infsup}, namely 
\begin{equation}
\label{eq:tutu}
\inf_{\overline{A}\in{\cal S}} \ \Phi_\eps(\overline{A},f).
\end{equation}
Let $f\in\Ltwon$ be fixed. It can be easily shown, starting from~\eqref{eq:Phi} and using the linearity of both the divergence and inverse laplacian operators, that
\begin{equation} \label{eq:Phi.tens}
\Phi_\eps(\overline{A},f) = \frac{1}{2} \sum_{1\leq i,j,k,l\leq d} \left[ \mathbb{B}_\eps(f)\right]_{i,j,k,l} \, \overline{A}_{i,j} \, \overline{A}_{k,l}-\sum_{1\leq i,j\leq d} \left[B_\eps(f)\right]_{i,j} \, \overline{A}_{i,j}+b(f),
\end{equation}
where the fourth-order tensor $\mathbb{B}_\eps(f)$, the matrix $B_\eps(f)$ and the scalar $b(f)$, which all depend on $f$, are given, for integers $1\leq i,j,k,l\leq d$, by
\begin{eqnarray*}
\left[\mathbb{B}_\eps(f)\right]_{i,j,k,l}
& = &
\dis 2 \int_{\cal D} \left[ (-\Delta)^{-1}(\partial_{ij} u_\eps(f)) \right] \ \left[ (-\Delta)^{-1}(\partial_{kl} u_\eps(f))\right],
\\
\left[B_\eps(f)\right]_{i,j}
& = & 
\dis -2 \int_{\cal D} \left[ (-\Delta)^{-1}(\partial_{ij} u_\eps(f))\right] \ \left[ (-\Delta)^{-1}f\right],
\\
b(f)
& = &
\norm[\Ltwo]{(-\Delta)^{-1}f}^2.
\end{eqnarray*}
Practically, the $\inf$ problem~\eqref{eq:tutu} (with fixed $f$) is solved on the whole set $\R^{d \times d}_{\rm sym}$ of symmetric matrices, instead of considering the subset ${\cal S}$ of positive-definite symmetric matrices. Under this simplification, solving the $\inf$ problem~\eqref{eq:tutu} amounts to considering the linear system
\begin{equation}
\label{eq:lin.sys-pre}
\forall\,1\leq i, j\leq d, \quad \sum_{1\leq k,l\leq d} \left[\mathbb{B}_\eps(f)\right]_{i,j,k,l} \ \overline{A}_{k,l} = \left[B_\eps(f)\right]_{i,j}.
\end{equation}
This system is low-dimensional and inexpensive to solve. In our numerical experiments, we have observed that the problem~\eqref{eq:lin.sys-pre} always has a unique solution in $\R^{d \times d}_{\rm sym}$, for all the functions $f$ that our algorithm explores. In addition, this solution is in ${\cal S}$.

\subsection{Asymptotic consistency} \label{sse:as.cons}

We study here problem~\eqref{eq:infsup} in the limit of a vanishing parameter $\eps$. We introduce the notation
\begin{equation} \label{eq:Phi.sup}
\Phi_\eps(\overline{A}) = \sup_{f\in\Ltwon}\Phi_\eps(\overline{A},f).
\end{equation}
Note that $\Phi_\eps$ is nonnegative. Consequently, for any $\eps$, problem~\eqref{eq:infsup} admits a quasi-minimizer, namely a matrix $\overline{A}^\flat_\eps\in{\cal S}$ such that
\begin{equation} \label{eq:estim}
I_\eps\leq\Phi_\eps(\overline{A}^\flat_\eps)\leq I_\eps+\eps\leq\Phi_\eps(\overline{A})+\eps\qquad\text{for any }\overline{A}\in{\cal S}.
\end{equation}
The following proposition holds.

\begin{proposition}[Asymptotic consistency, periodic case] \label{pr:as.cons}
Consider problem~\eqref{eq:infsup}, that is 
$$
I_\eps = \inf_{\overline{A}\in{\cal S}}\sup_{f\in\Ltwon}\Phi_\eps(\overline{A},f).
$$ 
In the periodic setting, namely under the assumptions~\eqref{eq:ass.per} and~\eqref{eq:ass.bound.per}, the following convergence holds:
\begin{equation} \label{eq:lim.Ieps}
\lim_{\eps\to 0} I_\eps=0.
\end{equation}
Furthermore, for any sequence $\left\{ \overline{A}^\flat_\eps \in {\cal S} \right\}_{\eps>0}$ of quasi-minimizers of~\eqref{eq:infsup}, we have
\begin{equation} \label{eq:lim.Abareps}
\lim_{\eps\to 0}\overline{A}^\flat_\eps=A_\star.
\end{equation}
\end{proposition}

The proof of these results, which is postponed until Appendix~\ref{se:proof}, relies on two facts:
\begin{enumerate}
\item The homogenized matrix $A_\star\in{\cal S}_{\alpha,\beta}\subset{\cal S}$ can be used as a test-matrix in~\eqref{eq:estim}. In view of Lemma~\ref{le:test} below, it satisfies $\dis \lim_{\eps\to 0}\Phi_\eps(A_\star)=0$, which directly implies~\eqref{eq:lim.Ieps};
\item We show in Lemma~\ref{le:recons} below that there exist $d(d+1)/2$ right-hand sides $f_{\star,k}\in\Ltwon$ such that the knowledge of $f_{\star,k}$ and of $u_{\star,k}$ solution to~\eqref{eq:hom.0} with right-hand side $f_{\star,k}$, $1 \leq k \leq d \, (d+1)/2$, is sufficient to uniquely reconstruct the constant symmetric matrix $A_\star$. The proof of~\eqref{eq:lim.Abareps} relies on this argument and on~\eqref{eq:lim.Ieps}. We denote
\begin{equation} \label{eq:def_F}
{\cal F} = \Big\{ f_{\star,k}, \quad 1 \leq k \leq d(d+1)/2 \Big\}
\end{equation}
this set. 
\end{enumerate}
We do not know whether, for $\eps$ fixed, the infimum in~\eqref{eq:infsup} is attained, unless $\eps$ is sufficiently small (see Remark~\ref{rem:preuve_referee} in Appendix~\ref{sec:app_preuve} below). We will proceed throughout the article manipulating quasi-minimizers in the sense of~\eqref{eq:estim}.

\begin{remark}
\label{re:cras}
The analysis of the approach~\eqref{eq:infsup.old} introduced in~\cite{LBLeL:13} relies on the same arguments as the approach introduced here: Lemma~\ref{le:recons}, and the equivalent of Lemma~\ref{le:test} for the functional considered in~\cite{LBLeL:13}, that is $\dis \lim_{\eps \to 0} \Psi_\eps(A^\star) = 0$, where, for any $\overline{A} \in {\cal S}$, 
$$
\Psi_\eps(\overline{A}) = \sup_{f\in\Ltwon} \norm[\Ltwo]{u_\eps(f) - \overline{u}(f)}^2.
$$
\end{remark}

\begin{remark}
\label{re:general}
Note that the assumptions~\eqref{eq:ass.per} and~\eqref{eq:ass.bound.per} are not necessary to prove the results~\eqref{eq:lim.Ieps} and~\eqref{eq:lim.Abareps}. All that needs to be assumed is that the sequence of matrices $\{ A_\eps \}_{\eps>0}$ converges, in the sense of homogenization, to a {\em constant} and {\em symmetric} homogenized matrix $A_\star$. 
In that vein, we will see in Section~\ref{sse:stoch} below that the conclusions of Proposition~\ref{pr:as.cons} carry over to the specific stochastic case we consider there.
\end{remark}

\begin{remark} \label{re:infmax}
Consider the set ${\cal F}$ defined by~\eqref{eq:def_F}, and let
\begin{equation} \label{eq:infmax}
I_\eps^{\rm max} = \inf_{\overline{A}\in{\cal S}} \, \max_{f\in{\cal F}} \, \Phi_\eps(\overline{A},f).
\end{equation}
This problem is, in principle, easier to solve than~\eqref{eq:infsup}, as we replaced the supremum over $f\in\Ltwon$ by a maximization over the finite set ${\cal F}$. Let $\dis \Phi^{\rm max}_\eps(\overline{A}) = \max_{f\in{\cal F}}\Phi_\eps(\overline{A},f)$. For any quasi-minimizer $\overline{A}^{\rm max,\flat}_\eps\in{\cal S}$ of~\eqref{eq:infmax}, we have
$$
I_\eps^{\rm max}\leq\Phi^{\rm max}_\eps(\overline{A}^{\rm max,\flat}_\eps)\leq I_\eps^{\rm max}+\eps\leq\Phi^{\rm max}_\eps(A_\star)+\eps\leq\Phi_\eps(A_\star)+\eps.
$$
Since $\dis \lim_{\eps\to 0}\Phi_\eps(A_\star)=0$, we get that $\dis \lim_{\eps\to 0}I_\eps^{\rm max}=0$. In addition, one can show that $\overline{A}^{\rm max,\flat}_\eps\to A_\star$ as $\eps\to 0$ (we refer to Remark~\ref{re:proof} below for details). Similarly to~\eqref{eq:infsup}, the approach~\eqref{eq:infmax} is therefore asymptotically consistent. Note however that, in practice, the functions of the set ${\cal F}$ defined by~\eqref{eq:def_F} are unknown.
\end{remark}

\medskip

We note that Proposition~\ref{pr:as.cons} provides, in the setting described in Section~\ref{sse:ass}, a characterization of the homogenized matrix which is an alternative to the standard characterization of homogenization theory. To the best of our knowledge, this characterization has never been made explicit in the literature. 

\subsection{Approximation of $u_\eps$ in the $H^1$ norm} \label{sse:hone}

As a consequence of Proposition~\ref{pr:as.cons}, we note that $\overline{u}_\eps$, solution to~\eqref{eq:bar.0} with matrix $\overline{A}_\eps$, is an accurate approximation of $u_\eps$ in the $L^2$ norm, but not in the $H^1$ norm. Indeed, when $\eps$ goes to zero, $\overline{A}_\eps$ converges to $A_\star$. Hence, for $\eps$ sufficiently small, $\overline{u}_\eps$ is an accurate $H^1$-approximation of $u_\star$ solution to~\eqref{eq:hom.0}. In addition, from homogenization theory, we know that $u_\star$ is an accurate $L^2$-approximation of $u_\eps$. This implies that $\dis \lim_{\eps \to 0} \| \overline{u}_\eps - u_\eps \|_{L^2({\cal D})} = 0$. 

Note also that $u_\star$ and $u_\eps$ are not close to each other in the $H^1$ norm, and hence $\overline{u}_\eps$ is not an accurate approximation of $u_\eps$ in the $H^1$ norm. We present here an approach to reconstruct such an approximation.

\medskip

In many settings of homogenization theory (and in particular in the periodic setting we consider here), once the corrector problems are solved to compute the homogenized matrix, one can consider the two-scale expansion (truncated at the first-order)
\begin{equation} \label{eq:tse}
u_\eps^{1,\vec{\theta}}(\vec{x})
= 
u_\star(\vec{x}) + \eps \sum_{i=1}^d w^{\theta_i}_{\vec{e}_i}(\vec{x}/\eps) \, \partial_i u_\star(\vec{x}),
\end{equation}
where $w^{\theta_i}_{\vec{e}_i}$ is the unique solution with mean value $\theta_i \in \R$ to the periodic corrector equation~\eqref{eq:corr.per} for $\vec{p}=\vec{e}_i$. It is well-known that this two-scale expansion approximates $u_\eps$ in the $H^1$ norm, in the sense that, under some regularity assumptions (see e.g.~\cite{allaire-amar}), we have
\begin{equation} \label{eq:tse.est}
\norm[H^1({\cal D})]{u_\eps-u_\eps^{1,\vec{\theta}}}\leq C\,\sqrt{\eps}
\end{equation}
for a constant $C$ independent of $\eps$. 

\begin{remark}
\label{rem:renormal}
From the theoretical perspective, the mean value $\vec{\theta}$ of the correctors is irrelevant, and the estimate~\eqref{eq:tse.est} holds for any fixed $\vec{\theta}$. From the numerical perspective, the error $\norm[H^1({\cal D})]{u_\eps-u_\eps^{1,\vec{\theta}}}$ slightly depends on $\vec{\theta}$, in particular when $\eps$ is not asymptotically small. In view of the numerical tests described in Section~\ref{se:num} below (see e.g.~\eqref{eq:per.l2}), we keep track of this parameter.
\end{remark}

Computing the gradient of~\eqref{eq:tse}, we deduce from~\eqref{eq:tse.est} that
\begin{equation} \label{eq:grad.ueps}
\grad u_\eps = C_\eps \, \grad u_\star + \text{h.o.t.},
\end{equation}
where the $d \times d$ matrix $C_\eps$ is given by
\begin{equation} \label{eq:Ceps}
\left[C_\eps\right]_{i,i} = 1 + \partial_i w_{\vec{e}_i}(\cdot/\eps),
\qquad \qquad
\left[C_\eps\right]_{i,j} = \partial_i w_{\vec{e}_j}(\cdot/\eps) \quad \text{if $j \neq i$.}
\end{equation}
Our idea for constructing an approximation of $\grad u_\eps$ is to mimick formula~\eqref{eq:grad.ueps} and seek an approximation under the form $\overline{C}_\eps \grad\overline{u}_\eps$. Once the best matrix $\overline{A}_\eps$ has been computed, we compute a surrogate $\overline{C}_\eps$ of $C_\eps$ by solving the least-squares problem
\begin{equation} \label{eq:hone}
\inf_{\overline{C}\in(\Ltwo)^{d \times d}} \ \ \sum_{r=1}^R\norm[\Ltwo^d]{\grad u_\eps(f_r)-\overline{C} \ \grad\overline{u}_\eps(f_r)}^2
\end{equation}
for a given number $R$ of right-hand sides. 

\medskip

In practice, the right-hand sides $f_r$ selected for~\eqref{eq:hone} are the first $R$ basis functions of the space $V^P_{\rm n}({\cal D})$ defined by~\eqref{eq:VPn}, with $R$ such that
$$
R \leq P.
$$
This choice makes the $H^1$-reconstruction an inexpensive post-processing procedure once the best matrix is computed, as we already have at our disposal $u_\eps(f_r)$ for $1\leq r\leq R$. 

\begin{remark}
In our numerical experiments, we have observed that the surrogate $\overline{C}_\eps$ that we construct is indeed oscillatory, and essentially periodic when $A_\eps$ is periodic. This is expected since $\overline{C}_\eps$ is meant to be an approximation of $C_\eps$.
\end{remark}

In practice, we independently identify each row of $\overline{C}_\eps$, by considering (for any $1 \leq i \leq d$) the least-squares problem 
$$
\inf_{\overline{c}^i \in (\Ltwo)^d} \ \ \sum_{r=1}^R \norm[\Ltwo]{\partial_i u_\eps(f_r) - \overline{c}^i \cdot \grad\overline{u}_\eps(f_r)}^2.
$$
We next define the matrix $\overline{C}_\eps$ by $\left[ \overline{C}_\eps \right]_{i,j} = \left[ \overline{c}^i_\eps \right]_j$. In our numerical experiments, the functions $u_\eps$ and $\overline{u}_\eps$ are approximated by $u_{\eps,h}$ and $\overline{u}_{\eps,h}$ using a $\mathbb{P}^1$ Finite Element Method, and $\overline{c}^i_\eps$ is searched as a piecewise constant function. The value of $\overline{c}^i_\eps$ on an element $T$ is defined by the problem
\begin{equation} \label{eq:min.hone}
\inf_{\overline{\vec{c}}^i_T \in \R^d} \ \ \sum_{r=1}^R \left| \left[\partial_i u_{\eps,h}(f_r)\right]_{\mid T} - \overline{\vec{c}}^i_T \cdot \left[\grad \overline{u}_{\eps,h}(f_r)\right]_{\mid T} \right|^2,
\end{equation}
where the restrictions of $\partial_i u_{\eps,h}$ and $\grad \overline{u}_{\eps,h}$ to any element $T$ are constant. This problem is ill-posed if $R < d$, since, in this case, there exists vectors in $\R^d$ orthogonal to all $\left[\grad \overline{u}_{\eps,h}(f_r)\right]_{\mid T}$, $1 \leq r \leq R$. We thus always take $R \geq d$. To avoid technicalities related to the $\mathbb{P}^1$ discretization of $\overline{u}_\eps$, only mesh elements {\em not} contiguous to the boundary of ${\cal D}$ are considered in the minimization~\eqref{eq:min.hone}.

\subsection{The stationary ergodic setting} \label{sse:stoch}

We have focused in Sections~\ref{sse:infsup},~\ref{sse:as.cons} and~\ref{sse:hone} on the periodic setting. We now briefly turn to the stochastic ergodic setting. We introduce the {\em modified} cost function $\Phi^{\rm sto}_\eps$ defined, for any $\overline{A} \in \R^{d \times d}_{\rm sym}$ and $f\in\Ltwo$, by
\begin{equation} \label{eq:Phi.sto}
\Phi^{\rm sto}_\eps(\overline{A},f) = \norm[\Ltwo]{(-\Delta)^{-1}\left[\div\left(\overline{A}\grad\mathbb{E}(u_\eps(f))\right)+f\right]}^2.
\end{equation}
Note that $\Phi^{\rm sto}_\eps$ is a deterministic quantity. The difference with the cost function $\Phi_\eps$ defined by~\eqref{eq:Phi} in a deterministic context is that $\Phi^{\rm sto}_\eps$ involves $\mathbb{E}(u_\eps(f))$ rather than $u_\eps(f)$.

We next amend the $\inf\sup$ problem~\eqref{eq:infsup} in the following way. For a given value of $\eps$, we look for a best {\em deterministic} matrix $\overline{A}_\eps\in{\cal S}$ that solves the problem
\begin{equation} \label{eq:infsup.sto}
I^{\rm sto}_\eps = \inf_{\overline{A}\in{\cal S}}\,\sup_{f\in\Ltwon}\Phi^{\rm sto}_\eps(\overline{A},f).
\end{equation}
All the considerations of Sections~\ref{sse:infsup},~\ref{sse:as.cons} and~\ref{sse:hone} carry over, up to minor adjustments, to the present stochastic setting. Under assumptions~\eqref{eq:ass.sta} and~\eqref{eq:ass.bound.sta}, asymptotic consistency can be proved for any sequence $\{ \overline{A}^\flat_\eps \in {\cal S} \}_{\eps>0}$ of quasi-minimizers of~\eqref{eq:infsup.sto}. The adaptation of the proof of Proposition~\ref{pr:as.cons} to the stochastic setting is straightforward. It relies on the fact that, for any $f\in\Ltwo$, $\mathbb{E}(u_\eps(f))$ is bounded in $H^1({\cal D})$. Indeed, using that $\alpha \leq A_\eps(\cdot,\omega) \leq \beta$ almost surely, we have $\dis \norm[H^1({\cal D})]{u_\eps(\cdot,\omega)} \leq \frac{C}{\alpha} \norm[L^1({\cal D})]{f}$ almost surely (where $C$ is a deterministic constant only depending on ${\cal D}$), hence $\mathbb{E} \left[ \norm[H^1({\cal D})]{u_\eps}^2 \right]$ is bounded. Using the Cauchy-Schwarz inequality, we infer that $\mathbb{E}(u_\eps(f))$ is indeed bounded in $H^1({\cal D})$. We eventually get that $\grad\mathbb{E}(u_\eps(f))$ weakly converges, and $\mathbb{E}(u_\eps(f))$ strongly converges, in $\Ltwo$ and when $\eps$ goes to zero, to $\grad u_\star(f)$ and $u_\star(f)$, respectively, where $u_\star(f)$ is the solution to~\eqref{eq:hom.0}. 

\medskip

The $H^1$-reconstruction procedure presented in Section~\ref{sse:hone} is adapted to the stationary ergodic setting as follows. It is known that, almost surely, $u_\eps(\cdot,\omega)$ weakly converges in $H^1({\cal D})$ towards $u_\star$ when $\eps$ goes to zero. As in the periodic setting, the correctors allow to obtain a strong convergence in $H^1({\cal D})$, in the sense that (see~\cite[Theorem 3]{PapaV:81})
\begin{equation}
\label{eq:orlando}
\lim_{\eps \to 0} \mathbb{E} \left[ \norm[H^1({\cal D})]{u_\eps(\cdot,\omega) - u_\eps^1(\cdot,\omega)}^2 \right] = 0,
\end{equation}
with 
\begin{equation}
\label{eq:orlando2}
u_\eps^1(\vec{x},\omega)
= 
u_\star(\vec{x}) + \eps \sum_{i=1}^d w_{\vec{e}_i}(\vec{x}/\eps,\omega) \, \partial_i u_\star(\vec{x}),
\end{equation}
where $w_{\vec{e}_i}$ is the unique solution with vanishing mean value to the stochastic corrector equation~\eqref{eq:corr.sta} for $\vec{p}=\vec{e}_i$ (in contrast to the periodic case, see Remark~\ref{rem:renormal}, we only consider here correctors with {\em vanishing} mean, for the sake of simplicity). 

The equations~\eqref{eq:orlando}--\eqref{eq:orlando2} imply that
$$
\mathbb{E} \left[ \grad u_\eps(\cdot,\omega) \right] = C_\eps \, \grad u_\star + \text{h.o.t.},
$$
where the $d \times d$ matrix $C_\eps$ is given by
\begin{equation} \label{eq:Ceps_sto}
\left[C_\eps \right]_{i,i} = 1 + \mathbb{E} \left[ \partial_i w_{\vec{e}_i}(\cdot/\eps,\omega) \right],
\qquad \qquad
\left[C_\eps \right]_{i,j} = \mathbb{E} \left[ \partial_i w_{\vec{e}_j}(\cdot/\eps,\omega) \right] \quad \text{if $j \neq i$.}
\end{equation}
We have chosen to look for an approximation of $\mathbb{E}(\grad u_\eps)$ as follows. Once the best matrix $\overline{A}_\eps$ has been computed, we compute a surrogate $\overline{C}_\eps$ of $C_\eps$ by solving the least-squares problem
\begin{equation} \label{eq:hone_sto1}
\inf_{\overline{C}\in(\Ltwo)^{d \times d}} \ \ \sum_{r=1}^R \norm[\Ltwo^d]{\grad \mathbb{E} \left[ u_\eps(f_r) \right] -\overline{C} \ \grad\overline{u}_\eps(f_r)}^2
\end{equation}
for a given number $R$ of right-hand sides, which are selected as in the periodic setting (see Section~\ref{sse:hone}). Eventually, $\mathbb{E} \left[ \grad u_\eps(\cdot,\omega) \right]$ is approximated by $\overline{C}_\eps \, \grad \overline{u}_\eps$.

\begin{remark} \label{re:expec}
Criteria~\eqref{eq:Phi.sto} and~\eqref{eq:hone_sto1} are arbitrary and selected upon practical considerations. Among the possible alternatives, we could have considered
$$
\Phi^{\rm sto}_\eps(\overline{A},f) = \mathbb{E} \left[ \norm[\Ltwo]{(-\Delta)^{-1}\left[\div\left(\overline{A}\grad u_\eps(f) \right)+f\right]}^2 \right]
$$
instead of~\eqref{eq:Phi.sto}, and a similar alternative for the reconstruction~\eqref{eq:hone_sto1}. 

We have not proceeded in any of these directions. Note also that, in~\cite{LBLeL:13}, we defined the minimization problems $\omega$ by $\omega$ and next took the expectation of the results. Of course, considering expectations in the cost functions results in significant computational savings, besides actually improving accuracy and robustness.
\end{remark}

\section{Implementation details to solve~\eqref{eq:infsup.sto}} \label{se:impl}

We detail here how problem~\eqref{eq:infsup.sto}, in the stationary ergodic setting, can be efficiently solved in practice. Problem~\eqref{eq:infsup}, in the periodic setting, is actually simpler to solve, and we skip the easy adaptation to that case. 

The minimizer of~\eqref{eq:infsup.sto} is denoted by $\overline{A}_{\eps,h}^{P,M}$, where $h\ll\eps$ denotes the size of a mesh ${\cal T}_h = \{T\}$ of the domain ${\cal D}$, $P$ denotes the dimension of the subspace $V_{\rm n}^P({\cal D})$ of $\Ltwon$ used to approximate the $\sup$ problem (see~\eqref{eq:VPn}), and $M\in\N^\star$ denotes the number of Monte Carlo realizations used to approximate $\mathbb{E}(u_\eps)$ in~\eqref{eq:Phi.sto}.

\medskip

\noindent The algorithm consists of three steps:
\begin{enumerate}
\item Compute an approximation of $\Big\{ \mathbb{E}[u_\eps(f_p)] \Big\}_{1\leq p\leq P}$ (see Section~\ref{ssse:expec}). This is the most expensive step, as $M \times P$ oscillatory problems of the type~\eqref{eq:osc} are to be solved.
\item Compute an approximation of $(-\Delta)^{-1} f_p$ and of $\dis \left\{ (-\Delta)^{-1} \left( \partial_{ij} \mathbb{E}[u_\eps(f_p)]\right) \right\}_{1\leq i,j\leq d}$, for any $1\leq p\leq P$ (see Section~\ref{ssse:precomp}). This amounts to solving $P \left( 1+ d(d+1)/2 \right)$ problems with constant coefficients.
\item Solve problem~\eqref{eq:infsup.sto} iteratively (see Section~\ref{ssse:sol}). Each iteration involves diagonalizing a $P \times P$ matrix and solving a linear system with $d(d+1)/2$ unknowns. The cost of this third step is negligible.
\end{enumerate}
We now successively detail these three steps.

\subsection{Approximation of $\Big\{ \mathbb{E}[u_\eps(f_p)] \Big\}_{1\leq p\leq P}$} \label{ssse:expec}

For any basis function $f_p$ of $V^P_{\rm n}({\cal D})$, $1\leq p\leq P$, we approximate $\mathbb{E}[u_\eps(f_p)]$ by the empirical mean
\begin{equation} \label{eq:expec}
u_{\eps,h}^M(f_p) = \frac{1}{M}\sum_{m=1}^M u_{\eps,h}(f_p;\omega_m),
\end{equation}
where, for $1 \leq m \leq M$, $u_{\eps,h}(f_p;\omega_m)$ is the $\mathbb{P}^1$ approximation on ${\cal T}_h$ of $u_\eps(f_p;\omega_m)$, unique solution to~\eqref{eq:osc} with the oscillatory matrix-valued coefficient $A_\eps(\cdot,\omega_m)$ and the right-hand side~$f_p$.

To compute~\eqref{eq:expec} for all $1\leq p\leq P$, one has to (i) assemble $M$ random stiffness matrices, (ii) assemble $P$ deterministic right-hand sides, and (iii) solve $M\times P$ linear systems. This step is the only one involving Monte Carlo computations, and is therefore the most expensive part of the whole procedure.

\subsection{Precomputation of tensorial quantities} \label{ssse:precomp}

Once the computations of Section~\ref{ssse:expec} have been performed, we assemble some tensors that are needed to efficiently solve the $\sup$ and $\inf$ problems involved in~\eqref{eq:infsup.sto}.

We first compute, for any $1\leq p\leq P$, the approximations $z_h(f_p)$ and $\left\{ z^{M,ij}_{\eps,h}(f_p) \right\}_{1\leq i,j\leq d}$ on ${\cal T}_h$ of $(-\Delta)^{-1} f_p$ and $\dis \left\{ (-\Delta)^{-1} \left( \partial_{ij} \mathbb{E}[u_\eps(f_p)]\right) \right\}_{1\leq i,j\leq d}$. In particular, $z^{M,ij}_{\eps,h}(f_p)$ is such that, for any $\mathbb{P}^1$ function $w_h$ on ${\cal T}_h$ that vanishes on $\partial {\cal D}$, 
\begin{equation*} 
\int_{\cal D} \grad z^{M,ij}_{\eps,h}(f_p) \cdot \grad w_h = -\int_{\cal D} \partial_j \left[u_{\eps,h}^M(f_p)\right]\;\partial_i w_h.
\end{equation*}
Note that the following symmetry identity holds: $z^{M,ij}_{\eps,h}(f_p)=z^{M,ji}_{\eps,h}(f_p)$. 

We next assemble, for all integers $1\leq i,j,k,l\leq d$ and $1\leq p,q\leq P$, the quantities
\begin{eqnarray} 
\label{eq:Btens.app}
\left[{\cal K}_{\eps,h}^M\right]_{i,j,k,l,p,q} & = & 2 \int_{\cal D} z_{\eps,h}^{M,ij}(f_p) \, z_{\eps,h}^{M,kl}(f_q),
\\
\label{eq:Bmat.app}
\left[\mathbb{K}_{\eps,h}^M\right]_{i,j,p,q} & = & -\int_{\cal D} z_{\eps,h}^{M,ij}(f_p) \, z_h(f_q), 
\\
\label{eq:bscal.app}
\left[K_h\right]_{p,q}& = & \int_{\cal D} z_h(f_p) \, z_h(f_q).
\end{eqnarray}
We emphasize that the cost of this step depends on $P$ but is independent of the number $M$ of Monte Carlo realizations, and thus small in comparison to the cost of the operations described in Section~\ref{ssse:expec} for typical values of $M$ and $P$ (in the numerical results reported on in Section~\ref{se:num}, we have worked with $M=100$ and $P \leq 9$).

\subsection{Solution of the fully discrete problem} \label{ssse:sol}

\subsubsection{Formulation}

At this stage, the original problem~\eqref{eq:infsup.sto} has been approximated by its fully discrete version
\begin{equation} \label{eq:infsup.app}
I^{P,M}_{\eps,h} = \inf_{\overline{A}\in{\cal S}}\,\sup_{\vec{c} \in \R^P, \, |\vec{c}|^2=1}\Phi^{P,M}_{\eps,h}(\overline{A},\vec{c}),
\end{equation}
where, for any $\overline{A} \in \R^{d \times d}_{\rm sym}$ and $\vec{c}= \{ c_p \}_{1\leq p\leq P} \in \R^P$,
\begin{equation} 
\label{eq:Phi.app}
\Phi^{P,M}_{\eps,h}(\overline{A},\vec{c}) = \norm[\Ltwo]{\sum_{p=1}^Pc_p\left(\sum_{1\leq i,j\leq d}\overline{A}_{i,j}\,z^{M,ij}_{\eps,h}(f_p)+z_h(f_p)\right)}^2.
\end{equation}
Problem~\eqref{eq:infsup.app} is solved by iteratively considering the problem
\begin{equation} 
\label{eq:tutu1}
\sup_{\vec{c} \in \R^P, \, |\vec{c}|^2=1} \Phi^{P,M}_{\eps,h}(\overline{A},\vec{c})
\end{equation}
with $\overline{A}\in{\cal S}$ fixed, and the problem
\begin{equation} 
\label{eq:tutu2}
\inf_{\overline{A}\in{\cal S}} \Phi^{P,M}_{\eps,h}(\overline{A},\vec{c})
\end{equation}
with $\vec{c} \in \R^P$ fixed. We successively explain how we solve the $\sup$ problem~\eqref{eq:tutu1} (for $\overline{A}\in{\cal S}$ fixed), the $\inf$ problem~\eqref{eq:tutu2} (for $\vec{c}\in\R^P$ fixed), and next describe the iterative algorithm that we have implemented to solve~\eqref{eq:infsup.app}.

\subsubsection{The $\sup$ problem~\eqref{eq:tutu1}}

Let $\overline{A}\in{\cal S}$ be fixed. One can easily observe that
\begin{equation*} 
\Phi^{P,M}_{\eps,h}(\overline{A},\vec{c})=\vec{c}^T\,G_{\eps,h}^M(\overline{A})\,\vec{c},
\end{equation*}
where $G_{\eps,h}^M(\overline{A})$ is a symmetric, positive semi-definite, $P\times P$ matrix which can be assembled at no additional cost using the precomputed quantities defined in~\eqref{eq:Btens.app}--\eqref{eq:Bmat.app}--\eqref{eq:bscal.app} (see Appendix~\ref{sec:appB} for its exact expression). Solving the $\sup$ problem~\eqref{eq:tutu1} (with fixed matrix $\overline{A}$) hence amounts to finding a normalized eigenvector in $\R^P$ associated with the largest eigenvalue of the matrix $G_{\eps,h}^M(\overline{A})$. This is reminiscent of the eigenvalue problem discussed in Section~\ref{ssse:sup}. Practically, this eigenvector is computed using the power method. The cost of such a computation is negligible, owing to the small size of the matrix $G_{\eps,h}^M(\overline{A})$ (recall that $P$ is typically small in comparison to $M$). We denote by $\vec{c}(\overline{A})$ its solution and hence have
\begin{equation} \label{eq:Phi.G.bis}
\sup_{\vec{c}\in\R^P,\,|\vec{c}|^2=1}\Phi^{P,M}_{\eps,h}(\overline{A},\vec{c})=\vec{c}(\overline{A})^T\,G_{\eps,h}^M(\overline{A})\,\vec{c}(\overline{A}).
\end{equation} 

\subsubsection{The $\inf$ problem~\eqref{eq:tutu2}} \label{ssse:inf.app}

Let $\vec{c}\in\R^P$, $|\vec{c}|^2=1$, be fixed. We observe that
$$
\Phi^{P,M}_{\eps,h}(\overline{A},\vec{c}) = \frac{1}{2} \sum_{1\leq i,j,k,l\leq d} \left[\mathbb{B}_{\eps,h}^{P,M}(\vec{c})\right]_{i,j,k,l} \ \overline{A}_{i,j} \ \overline{A}_{k,l}-\sum_{1\leq i,j\leq d} \left[B_{\eps,h}^{P,M}(\vec{c})\right]_{i,j} \ \overline{A}_{i,j} + b_h^P(\vec{c}),
$$
where $\mathbb{B}_{\eps,h}^{P,M}(\vec{c})$ is a $d\times d\times d\times d$ fourth-order tensor, $B_{\eps,h}^{P,M}(\vec{c})$ is a $d\times d$ matrix and $b_h^P(\vec{c})$ is a scalar that can all be assembled at no additional cost using the precomputed quantities defined in~\eqref{eq:Btens.app}--\eqref{eq:Bmat.app}--\eqref{eq:bscal.app} (see Appendix~\ref{sec:appB} for their exact expressions). We recognize in $\Phi^{P,M}_{\eps,h}$ the discrete equivalent of~\eqref{eq:Phi.tens}. The $\inf$ problem~\eqref{eq:tutu2} (with fixed eigenvector $\vec{c}$) is in practice solved as explained in Section~\ref{ssse:inf}, by considering the linear system (see~\eqref{eq:lin.sys-pre})
\begin{equation} 
\label{eq:lin.sys}
\forall\,1\leq i, j\leq d, \quad \sum_{1\leq k,l\leq d} \left[ \mathbb{B}_{\eps,h}^{P,M}(\vec{c})\right]_{i,j,k,l} \ \overline{A}_{k,l} = \left[B_{\eps,h}^{P,M}(\vec{c}) \right]_{i,j}.
\end{equation}

\subsubsection{Iterative algorithm}

In the above description, we have considered either the $\sup$ problem (on $\vec{c}$, with fixed $\overline{A}$) or the $\inf$ problem (on $\overline{A}$, for fixed $\vec{c}$) involved in~\eqref{eq:infsup.app}. We now assemble these two building blocks to build an algorithm to solve~\eqref{eq:infsup.app}. Introducing
\begin{equation} \label{eq:Phi.sup.app}
\Phi_{\eps,h}^{P,M}(\overline{A}) = \sup_{\vec{c}\in\R^P,\,|\vec{c}|^2=1}\Phi_{\eps,h}^{P,M}(\overline{A},\vec{c}),
\end{equation}
we recast~\eqref{eq:infsup.app} as
\begin{equation} \label{eq:infsup.app.bis}
I^{P,M}_{\eps,h} = \inf_{\overline{A}\in{\cal S}}\Phi^{P,M}_{\eps,h}(\overline{A}).
\end{equation}
We have seen (see~\eqref{eq:Phi.G.bis}) that $\Phi_{\eps,h}^{P,M}(\overline{A})=\vec{c}(\overline{A})^T\,G_{\eps,h}^M(\overline{A})\,\vec{c}(\overline{A})$, where $\vec{c}(\overline{A})$ is an eigenvector of the matrix $G_{\eps,h}^M(\overline{A})$. One can easily prove that, for any $1\leq i,j\leq d$,
$$
\left[\grad_{\overline{A}}\Phi_{\eps,h}^{P,M}(\overline{A})\right]_{i,j} = \vec{c}(\overline{A})^T \ \partial_{\overline{A}_{i,j}} G_{\eps,h}^M(\overline{A}) \ \vec{c}(\overline{A}),
$$
which reads, using the expressions~\eqref{eq:Gmat},~\eqref{eq:Ktens} and~\eqref{eq:Kmat} of $G_{\eps,h}^M$, $\mathbb{B}_{\eps,h}^{P,M}$ and $B_{\eps,h}^{P,M}$ given in Appendix~\ref{sec:appB}, as
\begin{equation} \label{eq:grad}
\left[\grad_{\overline{A}}\Phi_{\eps,h}^{P,M}(\overline{A})\right]_{i,j} = \sum_{1\leq k,l\leq d} \left[\mathbb{B}_{\eps,h}^{P,M}(\vec{c}(\overline{A}))\right]_{i,j,k,l} \ \overline{A}_{k,l} - \left[B_{\eps,h}^{P,M}(\vec{c}(\overline{A}))\right]_{i,j}.
\end{equation}
Let $0<\mu<1$. In practice, we iterate as follows to solve problem~\eqref{eq:infsup.app.bis}. Let $n\in\N$ and $\overline{A}^n\in{\cal S}$.
\begin{enumerate}
\item We compute $\vec{c}^n = \vec{c}(\overline{A}^n)$ solution to the $\sup$ problem~\eqref{eq:Phi.sup.app} with fixed matrix $\overline{A}^n$.
\item We compute $\overline{A}^{n+1}_\flat \in \R^{d \times d}_{\rm sym}$ solution to the linear system~\eqref{eq:lin.sys} with fixed eigenvector $\vec{c}^n$. As pointed out above, we assume that $\overline{A}_\flat^{n+1}$ belongs to the convex subset ${\cal S}$ of $\R^{d \times d}_{\rm sym}$. It has always been the case in our numerical experiments.
\item We define the next iterate as
\begin{equation} \label{eq:iter}
\overline{A}^{n+1} = (1-\mu)\,\overline{A}^n+\mu\,\overline{A}^{n+1}_\flat.
\end{equation}
For the numerical results reported on in Section~\ref{se:num}, we have worked with $\mu \leq 0.1$.
\end{enumerate}
Since $\overline{A}^{n+1}$ is a convex combination of $\overline{A}^n \in {\cal S}$ and $\overline{A}^{n+1}_\flat \in {\cal S}$, we have $\overline{A}^{n+1} \in {\cal S}$. The iterations are initialized using, say,
$$
\overline{A}^0 = \mathbb{E}\left(\frac{1}{|{\cal D}|}\int_{\cal D} A_\eps(\vec{x},\cdot)\,{\rm d}\vec{x}\right).
$$

\medskip

Let us briefly explain, at least formally, why the algorithm defined above enables to find a minimizer of~\eqref{eq:infsup.app.bis}. We assume the linear system~\eqref{eq:lin.sys} to be invertible, and we denote by $\left[\mathbb{B}^{P,M}_{\eps,h}(\vec{c})\right]^{-1}$ its formal inverse. Since $\overline{A}^{n+1}_\flat$ is defined as the solution to~\eqref{eq:lin.sys} with eigenvector $\vec{c}^n$, we infer from~\eqref{eq:lin.sys} and~\eqref{eq:grad} that
$$
\mathbb{B}^{P,M}_{\eps,h}(\vec{c}^n)\overline{A}^{n+1}_\flat=B_{\eps,h}^{P,M}(\vec{c}^n)=\mathbb{B}^{P,M}_{\eps,h}(\vec{c}^n)\overline{A}^n-\grad_{\overline{A}}\Phi_{\eps,h}^{P,M}(\overline{A}^n),
$$
and thus
$$
\overline{A}^{n+1}_\flat=\overline{A}^n - \left[\mathbb{B}^{P,M}_{\eps,h}(\vec{c}^n)\right]^{-1}\grad_{\overline{A}}\Phi_{\eps,h}^{P,M}(\overline{A}^n).
$$
The iteration~\eqref{eq:iter} can be recast under the form
$$
\overline{A}^{n+1}=\overline{A}^n-\mu \, \left[\mathbb{B}^{P,M}_{\eps,h}(\vec{c}^n)\right]^{-1} \grad_{\overline{A}}\Phi_{\eps,h}^{P,M}(\overline{A}^n).
$$
This is a quasi-Newton algorithm for the minimization of the function $\overline{A}\mapsto\Phi_{\eps,h}^{P,M}(\overline{A})$, with a fixed step size $\mu$ and where the Hessian of $\Phi_{\eps,h}^{P,M}$ with respect to $\overline{A}$ is approximated by $\mathbb{B}^{P,M}_{\eps,h}$.

Note that each iteration of the algorithm is inexpensive in comparison with the cost of the operations described in Sections~\ref{ssse:expec} and~\ref{ssse:precomp}. Consequently, there is no real advantage in improving the optimization algorithm~\eqref{eq:iter} (e.g.~by optimizing the value of $\mu$ by a line search).

\section{Numerical results} \label{se:num}

As pointed out in Section~\ref{se:intro}, our approach targets practical situations where the information on the oscillatory coefficients in the equation may be incomplete, and thus the other available approaches cannot be applied. It is nevertheless a legitimate question to investigate how our approach performs on standard test-cases in the periodic and stationary ergodic settings, and how it compares with the classical homogenization approach for small values of $\eps$. As already pointed out in Section~\ref{sec:outline}, and as detailed below (see Section~\ref{ssse:eps.inf.eval}), the aim of the numerical tests is different in the periodic setting and in the stochastic setting. It is also different if $\eps$ is asymptotically small or if $\eps$ takes larger values.

This section is organized as follows. In Section~\ref{sse:cases}, we introduce the periodic and the stationary ergodic test cases considered. In Section~\ref{sse:eps.inf}, we present the numerical results obtained in the case of small values of $\eps$. In Section~\ref{sse:eps.sup}, we address the case of larger values of $\eps$.

\subsection{Test-cases} \label{sse:cases}

We let $d=2$ and the domain ${\cal D}$ be the unit square $(0,1)^2$. We fix the value of the parameter $\overline{\eps}$ to ${\rm size}({\cal D})/10 = 10^{-1}$.

\subsubsection{Periodic setting}

We consider the test-case introduced in~\cite{LBLeL:13}, namely
\begin{equation} \label{eq:per.osc.1}
A_\eps(x,y)=A^{\rm per}(x/\eps,y/\eps),
\end{equation}
with $A^{\rm per}$ a $\mathbb{Z}^2$-periodic symmetric matrix field given by
\begin{equation} \label{eq:per.osc.2}
\begin{alignedat}{1}
\left[A^{\rm per}(x,y)\right]_{1,1}&=2+\frac{1}{2\pi}(\sin(2\pi x)+\sin(2\pi y)),
\\
\left[A^{\rm per}(x,y)\right]_{1,2}&=\frac{1}{2\pi}(\sin(2\pi x)+\sin(2\pi y)), 
\\
\left[A^{\rm per}(x,y)\right]_{2,2}&=1+\frac{1}{2\pi}(\sin(2\pi x)+\sin(2\pi y)).
\end{alignedat}
\end{equation}
The coefficients of the corresponding homogenized matrix (obtained by solving the periodic corrector problem~\eqref{eq:corr.per} on a very fine mesh) are
\begin{equation} \label{eq:per.star}
[A_\star]_{1,1}\approx 1.9806,\qquad [A_\star]_{1,2} = [A_\star]_{2,1} \approx -0.019345,\qquad [A_\star]_{2,2} \approx 0.98065.
\end{equation}

\subsubsection{Stationary ergodic setting}

We consider the random checkerboard test-case (studied e.g. in~\cite{LBLeL:13}), namely
\begin{equation} \label{eq:sta.osc.1}
A_\eps(x,y,\omega)=a^{\rm sto}(x/\eps,y/\eps,\omega)\,{\rm Id}_2,
\end{equation}
with $a^{\rm sto}$ a discrete stationary field given by (recall that $Q=(0,1)^2$)
\begin{equation} \label{eq:sta.osc.2}
a^{\rm sto}(x,y,\omega)=\sum_{\vec{k}\in\mathbb{Z}^2}\mathbb{1}_{Q+\vec{k}}(x,y)X_{\vec{k}}(\omega),
\end{equation}
where the random variables $X_{\vec{k}}$ are i.i.d.~and such that $\mathbb{P}(X_{\vec{k}}=4)=\mathbb{P}(X_{\vec{k}}=16)=1/2$. An explicit expression for the homogenized matrix is known in that case:
\begin{equation} \label{eq:sta.star}
A_\star=8\,{\rm Id}_2.
\end{equation}

\subsection{Results in the case $\eps<\overline{\eps}$} \label{sse:eps.inf}

\subsubsection{Objectives in the periodic case and in the stochastic case} 
\label{ssse:eps.inf.eval}

In the regime $\eps<\overline{\eps}$, we know from Proposition~\ref{pr:as.cons} that our method can be seen as a practical variational approach for computing the homogenized matrix $A_\star$. The remaining question is whether this approach is efficient or not, and particularly, compared with the classical approach in homogenization.

Our approach (based on~\eqref{eq:infsup}--\eqref{eq:Phi}) requires solving the {\em highly oscillatory} equations~\eqref{eq:osc.0} set on the domain ${\cal D}$, for $P=d(d+1)/2$ right-hand sides. In the periodic setting, the classical homogenization approach requires solving $d$ {\em non-oscillatory} equations set on the unit cell $Q$. There is thus no hope to outperform the latter approach in terms of computational time. This setting is nonetheless considered as a validation and we investigate how our approach performs in terms of accuracy, for the approximation of the homogenized matrix, and for the approximation of $u_\eps$ in the $L^2$ and $H^1$ norms.

The real, discriminating, test-case for our approach is the stationary ergodic setting. Indeed, classical homogenization then requires solving equations that are set on a truncated approximation $Q^N=(-N,N)^d$ of an asymptotically infinitely large domain (see~\eqref{eq:corr.N} in Section~\ref{sse:comp.hom}). The coefficients of these equations vary at scale $1$. In that case, to hope for an accurate approximation of the homogenized matrix, one has to consider a meshsize $H\ll 1$. On the other hand, we consider a meshsize $h \ll \eps$ to solve the highly oscillatory equations (set on the domain ${\cal D}$) involved in our approach. We see that, up to an appropriate choice of the parameter $H$ such that
\begin{equation} \label{eq:H}
\frac{2N}{H} = \frac{\text{size}({\cal D})}{h},
\end{equation}
where $\text{size}({\cal D})$ is typically the diameter of ${\cal D}$, the classical homogenization approach and ours involve solving linear systems of the same size. The computational workload for the two approaches is thus of the same order of magnitude, although not identical. We have decided to enforce~\eqref{eq:H} and to relate $N$ in~\eqref{eq:corr.N} and $\eps$ in~\eqref{eq:osc} by
\begin{equation} \label{eq:N}
N={\rm size}({\cal D})/2\eps.
\end{equation}
Note that imposing~\eqref{eq:N} is equivalent to enforcing $\eps/h=1/H$. We then compare the two methods in terms of solution time and accuracy. Obviously, for the two methods, the same number $M$ of Monte Carlo realizations is used, and the same $M$ realizations are considered.

\begin{remark}
Another possibility would have been to impose $\eps/h=1/H$ and to adjust the size $N$ of $Q^N$ in~\eqref{eq:corr.N} so that both approaches exactly share the same workload. We did not pursue in that direction.
\end{remark}

The numerical experiments reported in Section~\ref{ssse:eps.inf.sta} show that, in the stochastic case, and for all the values of $\eps<\overline{\eps}$ that have been considered, the approximation of $A_\star$ obtained by the classical homogenization approach is slightly more accurate than that obtained with our approach. In contrast, our approach provides a better $L^2$-approximation and a better $H^1$-approximation of $\mathbb{E}(u_\eps)$. This is somewhat intuitive, as our approach is targeted toward the approximation of $u_\eps$ rather than $A_\star$. In terms of computational cost, our approach is slightly less expensive for moderately small values of $\eps$, and slightly more expensive for asymptotically small values of $\eps$ (in any cases, the ratio of costs remains close to 1, see Figure~\ref{fi:2} below).

\subsubsection{Choice of the numerical parameters}

We recall that the integer $M$ denotes the number of i.i.d. realizations used to approximate the expectation in the cost function~\eqref{eq:Phi.sto} (see~\eqref{eq:expec}). We also recall that the integer $P$ denotes the dimension of the set $V^P_{\rm n}({\cal D})$ (defined in~\eqref{eq:VPn}) that is used to approximate the space $L^2_{\rm n}({\cal D})$ in the $\sup$ problem. As explained in Section~\ref{ssse:sup}, we consider as basis functions of the set $V^P_{\rm n}({\cal D})$ the first $P$ normalized eigenvectors of the laplacian operator in the domain ${\cal D}$. Because of the simple geometry of ${\cal D}$, they are here analytically known. We take here $P=d \, (d+1)/2$, that is $P=3$, which is the minimum dimension of the search space $V^P_{\rm n}({\cal D})$. 

\subsubsection{Results in the periodic setting} \label{ssse:eps.inf.per}

We consider the parameters $\{ \eps_k \}_{0\leq k\leq 6}$ such that $\eps_0=0.4$ and $\eps_k=\eps_{k-1}/2$ for $1\leq k\leq 6$. The associated meshsizes are $\{ h_k \}_{0\leq k\leq 6}$ such that $h_k=\eps_k/r$ for $r\approx 43$, unless otherwise mentioned. We focus on the values $\{ \eps_k \}_{3\leq k\leq 6}$, for which we have $\eps_k<\overline{\eps}$.

\medskip

The error in the approximation of the homogenized matrix is defined by
\begin{equation} \label{eq:per.mat}
\verb?err_per_mat? = \left(\frac{\sum_{1\leq i,j\leq d}\left| \left[\overline{A}_{\eps,h}^P\right]_{i,j}- [A_\star]_{i,j} \right|^2}{\sum_{1\leq i,j\leq d} \left| [A_\star]_{i,j} \right|^2 } \right)^{1/2},
\end{equation} 
where $A_\star$ is taken equal to its reference value~\eqref{eq:per.star} and $\overline{A}^P_{\eps,h}$ is the best matrix computed by our approach. The numerical results are collected in Table~\ref{ta:1}. We observe that our approach provides an accurate approximation of the homogenized matrix. The accuracy of the approximation improves (in the limit of spatial resolution) as $\eps$ decreases.

\medskip

\begin{verbbox}err_per_mat\end{verbbox}
\begin{table}[h!]
\begin{center}
\begin{tabular}{|c||c|c|c|c|}
\hline
$\eps$ & $0.05$ & $0.025$ & $0.0125$ & $0.00625$ \\
\hline
\verb?err_per_mat? ($\eps/h\approx 43$) & $1.0145 \ 10^{-3}$ & $7.6477 \ 10^{-4}$ & $6.6613 \ 10^{-4}$ & $6.2881 \ 10^{-4}$ \\
\hline
\verb?err_per_mat? ($\eps/h\approx 86$) & $6.5399 \ 10^{-4}$ & $3.5074 \ 10^{-4}$ & $2.3749 \ 10^{-4}$ & {\rm X} \\
\hline
\end{tabular}
\end{center}
\caption{Approximation of $A_\star$ (\theverbbox) in function of $\eps$ (each line corresponds to a different value of the ratio $\eps/h$). The test cases with $\eps$ too small and $\eps/h$ too large are prohibitively expensive to perform. They are marked with an X.} \label{ta:1}
\end{table}

\bigskip

\noindent We now examine the approximation of $u_\eps$ in the $L^2$ norm. We denote by
\begin{itemize}
\item[$\bullet$] $u_{\eps,h}(f)$ the discrete solution to~\eqref{eq:osc.0} with the periodic oscillatory coefficient given by~\eqref{eq:per.osc.1}--\eqref{eq:per.osc.2} and the right-hand side $f$;
\item[$\bullet$] $u_{\star,h}(f)$ the discrete solution to~\eqref{eq:hom.0} with the homogenized matrix~\eqref{eq:per.star} and the right-hand side $f$;
\item[$\bullet$] $u_{\eps,h}^{1,\vec{\theta}}(f)$ the two-scale expansion (truncated at first-order) built from $u_{\star,h}(f)$ (see~\eqref{eq:tse}), where we use the periodic correctors solution to~\eqref{eq:corr.per};
\item[$\bullet$] $\overline{u}^P_{\eps,h}(f)$ the discrete solution to~\eqref{eq:bar.0} with the matrix $\overline{A}_{\eps,h}^P$ and the right-hand side $f$ (we recall that the matrix $\overline{A}_{\eps,h}^P$ has been computed using a small number $P$ of right-hand sides).
\end{itemize}
To assess the quality of the approximation of $u_{\eps,h}$ by $\widehat{u}^{\vec{\theta}}_h \in \left\{ u_{\star,h}, \ u_{\eps,h}^{1,\vec{\theta}}, \ \overline{u}^P_{\eps,h} \right\}$ in the $L^2$ norm, we define the criterion
\begin{equation} \label{eq:per.l2} 
\verb?err_per_L2? = \left( \frac{\dis \inf_{\vec{\theta}\in\mathbb{R}^2} \left[ \sup_{f\in V^{\cal Q}_{\rm n}({\cal D})} \norm[\Ltwo]{u_{\eps,h}(f) - \widehat{u}^{\vec{\theta}}_h(f)}^2 \right]}{\norm[\Ltwo]{u_{\eps,h} \left( \widehat{f}_\eps \right)}^2} \right)^{1/2}.
\end{equation}
Note that the supremum is taken over $f\in V^{\cal Q}_{\rm n}({\cal D})$, where ${\cal Q} \gg P$. We take ${\cal Q}=16$, and we have checked, in all the cases considered below, that our results do not significantly change for a larger value of ${\cal Q}$. The function $\widehat{f}_\eps\in V^{\cal Q}_{\rm n}({\cal D})$ denotes the argument of the $\inf \sup$ problem in the numerator of~\eqref{eq:per.l2}. We hence compare $u_\eps$ with its homogenized limit $u_\star$, its first-order two-scale expansion $u_\eps^{1,\vec{\theta}}$ (recall in this case that the correctors are defined up to an additive constant $\vec{\theta}$, over which we minimize the error in~\eqref{eq:per.l2}), and the approximation $\overline{u}^P_\eps$ provided by our approach. The numerical results are collected in Figure~\ref{fi:1}.

We observe that the solution associated with the best matrix we compute indeed converges towards the exact solution, in the $L^2$ norm. We however recall that, in the present periodic setting, computing $\overline{u}_{\eps,h}^P$ is much more expensive than computing $u_{\star,h}$ or $u_{\eps,h}^{1,\vec{\theta}}$. 

\medskip

\begin{verbbox}err_per_L2\end{verbbox}
\begin{figure}[h!]
\centering
\begin{tikzpicture}[scale=1.]
\begin{loglogaxis} [
xtick={0.05,0.025,0.0125,0.00625},
xticklabels={$0.05$,$0.025$,$0.0125$,$0.00625$}
]
\addplot[thick,mark=x,red] table[x=epsilon,y=err_l2_star]{dat/eps_inf/det/err_l2_43.dat};

\addplot[thick,mark=x,brown] table[x=epsilon,y=err_l2_star_corr]{dat/eps_inf/det/err_l2_43.dat};

\addplot[thick,mark=x,black] table[x=epsilon,y=err_l2_bar]{dat/eps_inf/det/err_l2_43.dat};
\end{loglogaxis}
\end{tikzpicture}
\caption{Approximation of $u_\eps$ in the $L^2$ norm (\theverbbox) by $u_{\star,h}$ (red), $u_{\eps,h}^{1,\vec{\theta}}$ (brown) and $\overline{u}^P_{\eps,h}$ (black) in function of $\eps$, for $h$ such that $\eps/h \approx 43$.} \label{fi:1}
\end{figure}

\bigskip

We next examine the $H^1$ error. For $f\in\Ltwo$, we denote by $C_{\eps,h}\grad u_{\star,h}(f)$ the discrete equivalent of $C_\eps\grad u_\star(f)$, the homogenization-based approximation of $\grad u_\eps(f)$, see~\eqref{eq:grad.ueps}--\eqref{eq:Ceps} in Section~\ref{sse:hone}. We recall that, in our approach, we seek an approximation of $\grad u_\eps(f)$ under the form $\overline{C}_{\eps}\grad\overline{u}_\eps(f)$ (see~\eqref{eq:hone}), the discrete equivalent of which is computed as $\overline{C}^R_{\eps,h}\grad\overline{u}^P_{\eps,h}(f)$. Recall that the integer $R$ is the number of right-hand sides used to define the least-squares minimization problem~\eqref{eq:min.hone} giving $\overline{C}^R_{\eps,h}$. Here, we take $R=P=3$. To assess the quality of the approximation of $\grad u_{\eps,h}$, we define, for $\widehat{C}_{\eps,h}\grad\widehat{u}_h \in \left\{ C_{\eps,h}\grad u_{\star,h}, \ \overline{C}^R_{\eps,h}\grad\overline{u}^P_{\eps,h}\right\}$, the criterion
\begin{equation} \label{eq:per.h1}
\verb?err_per_H1? = \left( \frac{ \dis \sup_{f\in V^{\cal Q}_{\rm n}({\cal D})} \norm[L^2({\cal D} \setminus {\cal B})]{\grad u_{\eps,h}(f) - \widehat{C}_{\eps,h} \grad \widehat{u}_h(f)}^2}{\norm[L^2({\cal D} \setminus {\cal B})]{\grad u_{\eps,h} \left( \widehat{f}_\eps \right)}^2} \right)^{1/2},
\end{equation}
where, here again, the supremum is taken over a space $V^{\cal Q}_{\rm n}({\cal D})$ much larger than $V^P_{\rm n}({\cal D})$ (we take ${\cal Q}=16$), and where $\widehat{f}_\eps\in V^{\cal Q}_{\rm n}({\cal D})$ denotes the argument of the $\sup$ problem. In~\eqref{eq:per.h1}, ${\cal B}$ represents the subset of ${\cal D}$ formed by the boundary elements of the discretization ${\cal T}_h$. We remove them in view of the discussion below~\eqref{eq:min.hone}. 
We thus compare $\grad u_\eps$ with its approximation $C_\eps \grad u_\star$ provided by the two-scale expansion and with the approximation $\overline{C}^R_\eps \ \grad\overline{u}^P_\eps$ provided by our approach. The numerical results are collected in Table~\ref{ta:2}.

\medskip

We observe that our approach provides an accurate $H^1$-approximation of $u_\eps$. As $\eps$ goes to zero, the surrogate we compute is (roughly) a first-order convergent approximation of $\grad u_\eps$ in the $L^2$ norm. As far as the homogenization-based approximation is concerned, we expect it to converge with order at least one half (see~\eqref{eq:tse.est}). This is what we observe in practice, as long as $\eps$ is not too small. Otherwise, the error due to the meshsize dominates, and the error~\eqref{eq:per.h1} does not decrease anymore when $\eps$ decreases. 

\begin{verbbox}err_per_H1\end{verbbox}
\begin{table}[h!]
\begin{center}
\begin{tabular}{|c||c|c|c|c|}
\hline
$\eps$ & $0.05$ & $0.025$ & $0.0125$ & $0.00625$ \\
\hline
\verb?err_per_H1? for $C_{\eps,h}\grad u_{\star,h}$ & $2.0906 \ 10^{-2}$ & $1.6461 \ 10^{-2}$ & $1.2513 \ 10^{-2}$ & {\rm X} \\
\hline
\verb?err_per_H1? for $\overline{C}^R_{\eps,h}\grad\overline{u}^P_{\eps,h}$ & $1.5550 \ 10^{-2}$ & $7.6055 \ 10^{-3}$ & $3.7549 \ 10^{-3}$ & {\rm X} \\
\hline
\end{tabular}
\end{center}
\caption{Approximation of $\grad u_\eps$ in the $L^2$ norm (\theverbbox) by $C_{\eps,h}\grad u_{\star,h}$ and $\overline{C}^R_{\eps,h}\grad\overline{u}^P_{\eps,h}$ in function of $\eps$, for $h$ such that $\eps/h \approx 86$. The test cases with $\eps$ too small are prohibitively expensive to perform. They are marked with an X.} \label{ta:2}
\end{table}

\subsubsection{Results in the stationary ergodic setting} \label{ssse:eps.inf.sta}

We consider the parameters $\{ \eps_k \}_{0\leq k\leq 5}$ such that $\eps_k=2^{-(k+1)}$ for $0\leq k\leq 5$. In agreement with formula~\eqref{eq:N}, we couple these parameters to the parameters $\{ N_k \}_{0\leq k\leq 5}$ (defining the domain on which we solve the corrector problems~\eqref{eq:corr.N}) such that $N_k=2^k$. The associated meshsizes $\{ h_k \}_{0\leq k\leq 5}$ and $\{ H_k \}_{0\leq k\leq 5}$ are computed respectively letting $h_k=\eps_k/r$ for $r\approx 27$ (unless otherwise stated) and using~\eqref{eq:H}. We focus on the values $\{ \eps_k \}_{3\leq k\leq 5}$ and $\{ N_k \}_{3\leq k\leq 5}$, for which we have $\eps_k<\overline{\eps}$. We consider $M=100$ Monte Carlo realizations.

\medskip

Before discussing the accuracy of our approach, we first compare its cost with that of the classical approach. We show on Figure~\ref{fi:2} the ratio of the time needed to compute $\overline{A}^{P,M}_{\eps,h}$ using our approach divided by the time needed to compute $A_{\star,H}^{N,M}$ by the classical homogenization approach. To compare the computational times, we make use of an implementation that does not exploit parallelism, and we solve the linear systems by means of an iterative solver. In view of Figure~\ref{fi:2}, for the choice of parameters discussed in Section~\ref{ssse:eps.inf.eval}, our method is slightly faster than the standard homogenization approach for values of $N$ up to approximately $14$. This observation can be explained as follows. For the number $M=100$ of Monte Carlo realizations that we consider, we can neglect, in our procedure, the cost of the precomputation and final optimization stages, in comparison to the Monte Carlo step (see Section~\ref{se:impl}). Hence, to compute $\overline{A}_{\eps,h}^{P,M}$, we have to (i) assemble $M=100$ stiffness matrices, (ii) assemble $P=3$ right-hand sides, and (iii) solve $P\times M=300$ linear systems. In contrast, to compute $A_{\star,H}^{N,M}$, one has to solve $d\times M=200$ approximate corrector equations~\eqref{eq:corr.N}, that is to say (i) assemble $M=100$ stiffness matrices, (ii) assemble $d\times M=200$ right-hand sides, and (iii) solve $d\times M=200$ linear systems. Consequently, our approach necessitates solving $100$ more linear systems, but assembling $200$ less right-hand sides, than the classical homogenization approach. This explains what we observe. When the value of $N$ is not too large, the assembly cost is higher than the inversion cost, and our approach is faster.

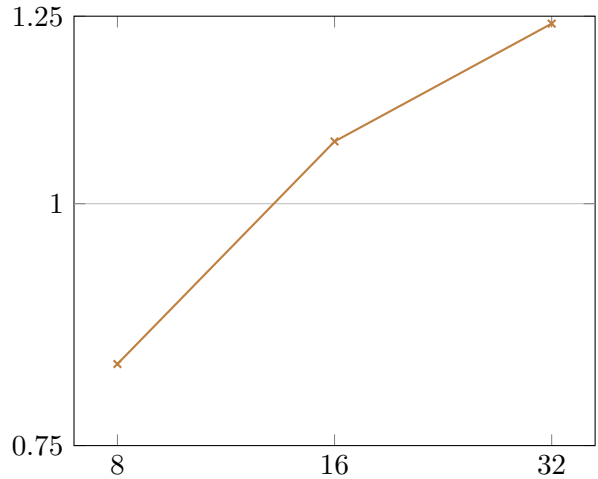
\begin{figure}[h!]
\centering
\begin{tikzpicture}[scale=1.]
\begin{loglogaxis} [
xtick={8,16,32},
xticklabels={$8$,$16$,$32$},
ymin=0.75,
ymax=1.25,
ytick={0.75,1.25},
yticklabels={$0.75$,$1.25$},
extra y ticks={1.},
extra y tick style={yticklabel={$1$}, grid=major}
]
\addplot[thick,mark=x,brown] table[x=N,y=optimdivbyhomo]{dat/eps_inf/sto/time_27_100.dat};
\end{loglogaxis}
\end{tikzpicture}
\caption{Ratio of the computational times between our approach and the classical homogenization approach, in function of $N$ (here $M=100$ and $\eps/h \approx 27$).} \label{fi:2}
\end{figure}

\bigskip

We adapt to the stationary ergodic setting the accuracy criteria~\eqref{eq:per.mat},~\eqref{eq:per.l2} and~\eqref{eq:per.h1} introduced in the periodic setting. The error in the approximation of the homogenized matrix is defined, for $\widehat{A}^M \in \left\{A_{\star,H}^{N,M}, \overline{A}_{\eps,h}^{P,M}\right\}$, by
\begin{equation*} 
\verb?err_sto_mat? = \left(\frac{\sum_{1\leq i,j\leq d} \left| \left[\widehat{A}^M\right]_{i,j} - [A_\star]_{i,j} \right|^2}{\sum_{1\leq i,j\leq d} \left| [A_\star]_{i,j} \right|^2 }\right)^{1/2},
\end{equation*}
where $A_\star$ is taken equal to the exact value~\eqref{eq:sta.star}. We recall that $A_{\star,H}^{N,M}$ is the practical approximation of $A_\star^{N,M}$ defined in~\eqref{eq:Astar.N.M}, and that our approach consists in computing the best matrix $\overline{A}^{P,M}_{\eps,h}$ following the procedure described in Section~\ref{sse:stoch}.

The numerical results are collected in Figure~\ref{fi:3}, for several choices of the meshsizes. We observe that the matrix we compute converges to the homogenized matrix as $N$ increases. However, for any value of $N$ in the range we consider, the approximation of $A_\star$ obtained by the classical homogenization approach is slightly more accurate than the one obtained with our approach. As shown on Figure~\ref{fi:2}, the former approach is as expensive as our approach for $N \approx 14$, and slightly less expensive for larger values of $N$.

\medskip

\begin{verbbox}err_sto_mat\end{verbbox}
\begin{figure}[h!]
\centering
\begin{tikzpicture}[scale=1.15]
\begin{loglogaxis} [
xtick={8,16,32},
xticklabels={$8$,$16$,$32$}
]
\addplot[dashed,mark=x,blue] table[x=N,y=err_mat_N]{dat/eps_inf/sto/err_mat_min_27_100.dat};
\addplot[dashed,mark=x,blue] table[x=N,y=err_mat_N]{dat/eps_inf/sto/err_mat_max_27_100.dat};
\addplot[thick,mark=x,blue] table[x=N,y=err_mat_N]{dat/eps_inf/sto/err_mat_mean_27_100.dat};
\addplot[dashed,mark=o,blue] table[x=N,y=err_mat_N]{dat/eps_inf/sto/err_mat_min_108_100.dat};
\addplot[dashed,mark=o,blue] table[x=N,y=err_mat_N]{dat/eps_inf/sto/err_mat_max_108_100.dat};
\addplot[thick,mark=o,blue] table[x=N,y=err_mat_N]{dat/eps_inf/sto/err_mat_mean_108_100.dat};

\addplot[dashed,mark=x,black] table[x=N,y=err_mat_optim]{dat/eps_inf/sto/err_mat_min_27_100.dat};
\addplot[dashed,mark=x,black] table[x=N,y=err_mat_optim]{dat/eps_inf/sto/err_mat_max_27_100.dat};
\addplot[thick,mark=x,black] table[x=N,y=err_mat_optim]{dat/eps_inf/sto/err_mat_mean_27_100.dat};
\addplot[dashed,mark=o,black] table[x=N,y=err_mat_optim]{dat/eps_inf/sto/err_mat_min_108_100.dat};
\addplot[dashed,mark=o,black] table[x=N,y=err_mat_optim]{dat/eps_inf/sto/err_mat_max_108_100.dat};
\addplot[thick,mark=o,black] table[x=N,y=err_mat_optim]{dat/eps_inf/sto/err_mat_mean_108_100.dat};
\end{loglogaxis}
\end{tikzpicture}
\caption{Approximation of $A_\star$ by the classical homogenization approach (blue) and by our approach (black) in function of $N$, for $M=100$ realizations. Since $M$ is finite, the error \theverbbox~is actually random. We compute it $100$ times. The thick line corresponds to the mean value over the $100$ computations of the error. The dashed lines show the $95\%$ confidence interval. Results obtained with $h$ such that $\eps/h \approx 27$ (resp. $\eps/h \approx 108$) are denoted with {\tt x} (resp. {\tt o}).} \label{fi:3}
\end{figure}
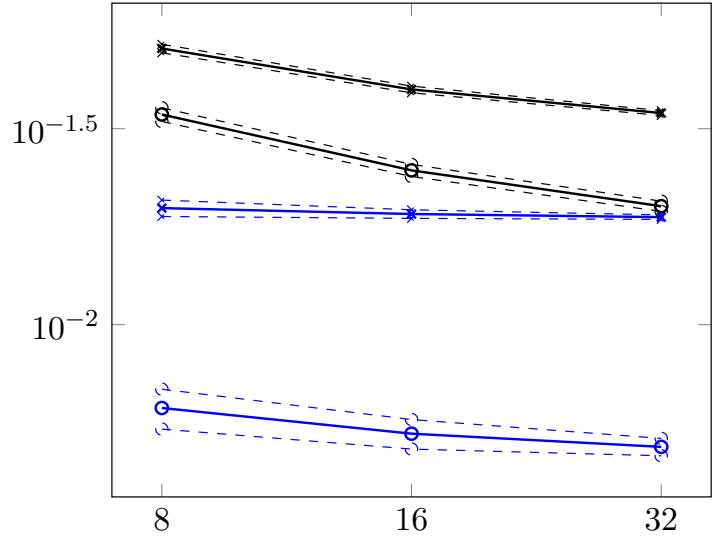

\bigskip

\noindent Turning to the approximation of $\mathbb{E}(u_\eps)$ in the $L^2$ norm, we denote by
\begin{itemize}
\item[$\bullet$] $u_{\eps,h}^M(f)$ the expectation, as defined in~\eqref{eq:expec}, of the discrete solutions to~\eqref{eq:osc} with the oscillatory coefficients given by~\eqref{eq:sta.osc.1}--\eqref{eq:sta.osc.2} and the right-hand side $f$;
\item[$\bullet$] $u_{\star,h}(f)$ the discrete solution to~\eqref{eq:hom.0} with the exact homogenized matrix~\eqref{eq:sta.star} and the right-hand side $f$ (note that the exact matrix is usually unknown);
\item[$\bullet$] $u_{\star,h}^{N,M}(f)$ the discrete solution to~\eqref{eq:hom.0} with the matrix $A_{\star,H}^{N,M}$ and the right-hand side $f$;
\item[$\bullet$] $\overline{u}^{P,M}_{\eps,h}(f)$ the discrete solution to~\eqref{eq:bar.0} with the matrix $\overline{A}_{\eps,h}^{P,M}$ and the right-hand side $f$.
\end{itemize}
The $M$ realizations of the field $A(\cdot,\omega)$ we consider to compute $u^M_{\eps,h}(f)$, $u_{\star,h}^{N,M}(f)$ and $\overline{u}^{P,M}_{\eps,h}(f)$ are identical. 

To assess the quality of the approximation of $u^M_{\eps,h}$ by $\widehat{u}_h \in \left\{ u_{\star,h}, \ u_{\star,h}^{N,M}, \ \overline{u}^{P,M}_{\eps,h}\right\}$ in the $L^2$ norm, we define the criterion
\begin{equation} \label{eq:sta.l2}
\verb?err_sto_L2? = \left( \frac{ \dis \sup_{f\in V^{\cal Q}_{\rm n}({\cal D})} \norm[\Ltwo]{u^M_{\eps,h}(f) - \widehat{u}_h(f)}^2}{\norm[\Ltwo]{u_{\eps,h}^M\left(\widehat{f}_\eps\right)}^2}\right)^{1/2}.
\end{equation}
As in the periodic case, the supremum is taken over $f\in V^{\cal Q}_{\rm n}({\cal D})$ with ${\cal Q}=16 \gg P$, and $\widehat{f}_\eps\in V^{\cal Q}_{\rm n}({\cal D})$ denotes the argument of the $\sup$ problem. The numerical results are collected in Figure~\ref{fi:4}, for several choices of the meshsizes and of the total number $M$ of realizations.

We observe that the solution associated with the best matrix we compute is a better $L^2$-approximation (for the range of parameters considered here) of $\mathbb{E}(u_\eps)$ than the solutions associated with the exact or approximate homogenized matrices. Again, due to the small number $P$ of right-hand sides we consider to compute $\overline{A}_{\eps,h}^{P,M}$, this good accuracy is not an immediate consequence of our practical procedure (it would have been if we had taken $P$ extremely large). We also observe that the accuracy of the three approximations $u_{\star,h}$, $u_{\star,h}^{N,M}$ and $\overline{u}_{\eps,h}^{P,M}$ improves when $h$ decreases or when $M$ increases, in somewhat a complex manner. In terms of cost, our approach is again less expensive than the classical approach for $N \leq 14$.

\medskip

\begin{verbbox}err_sto_L2\end{verbbox}
\begin{figure}[h!]
\centering
\begin{tikzpicture}[scale=1.3]
\begin{loglogaxis} [
xtick={8,16,32},
xticklabels={$8$,$16$,$32$}
]
\addplot[thick,mark=x,red] table[x=N,y=err_l2_star]{dat/eps_inf/sto/err_l2_27_100.dat};
\addplot[thick,mark=o,red] table[x=N,y=err_l2_star]{dat/eps_inf/sto/err_l2_108_100.dat};
\addplot[thick,mark=+,red] table[x=N,y=err_l2_star]{dat/eps_inf/sto/err_l2_27_400.dat};
\addplot[thick,mark=square,red] table[x=N,y=err_l2_star]{dat/eps_inf/sto/err_l2_54_400.dat};

\addplot[thick,mark=x,blue] table[x=N,y=err_l2_staN]{dat/eps_inf/sto/err_l2_27_100.dat};
\addplot[thick,mark=o,blue] table[x=N,y=err_l2_staN]{dat/eps_inf/sto/err_l2_108_100.dat};
\addplot[thick,mark=+,blue] table[x=N,y=err_l2_staN]{dat/eps_inf/sto/err_l2_27_400.dat};
\addplot[thick,mark=square,blue] table[x=N,y=err_l2_staN]{dat/eps_inf/sto/err_l2_54_400.dat};

\addplot[thick,mark=x,black] table[x=N,y=err_l2_bar]{dat/eps_inf/sto/err_l2_27_100.dat};
\addplot[thick,mark=o,black] table[x=N,y=err_l2_bar]{dat/eps_inf/sto/err_l2_108_100.dat};
\addplot[thick,mark=+,black] table[x=N,y=err_l2_bar]{dat/eps_inf/sto/err_l2_27_400.dat};
\addplot[thick,mark=square,black] table[x=N,y=err_l2_bar]{dat/eps_inf/sto/err_l2_54_400.dat};
\end{loglogaxis}
\end{tikzpicture}
\caption{Approximation of $\mathbb{E}(u_\eps)$ in the $L^2$ norm (\theverbbox) by $u_{\star,h}$ (red), $u_{\star,h}^{N,M}$ (blue) and $\overline{u}_{\eps,h}^{P,M}$ (black) in function of $N$ (curves with {\tt x}: $\eps/h \approx 27$ and $M=100$; curves with {\tt o}: $\eps/h \approx 108$ and $M=100$; curves with {\tt +}: $\eps/h \approx 27$ and $M=400$; curves with $\square$: $\eps/h \approx 54$ and $M=400$).} \label{fi:4}
\end{figure}

\bigskip

We next turn to the $H^1$-error. We denote by $C^{N,M}_{\eps,h}$ the approximation of the deterministic matrix $C_\eps$ defined by~\eqref{eq:Ceps_sto} by an empirical mean over $M$ realizations of the corrector functions, solution to~\eqref{eq:corr.N}:
$$
\left[ C^{N,M}_{\eps,h} \right]_{i,j} 
= 
\delta_{ij} + \frac{1}{M} \sum_{m=1}^M \partial_i w^N_{\vec{e}_j}(\cdot/\eps,\omega_m).
$$
For $f\in\Ltwo$, we denote by $C^{N,M}_{\eps,h} \grad u_{\star,h}(f)$ and $C^{N,M}_{\eps,h} \grad u^{N,M}_{\star,h}(f)$ the two discrete equivalents of $C_\eps \, \grad u_\star(f)$, the homogenization-based approximation of $\mathbb{E}\left(\grad u_\eps(f)\right)$, obtained by using the exact homogenized matrix~\eqref{eq:sta.star} and the matrix $A_{\star,H}^{N,M}$, respectively, to compute an approximation of $u_\star(f)$. In our approach, we seek a discrete approximation of $\mathbb{E}\left(\grad u_\eps\right)$ under the form $\overline{C}^{R,M}_{\eps,h} \grad \overline{u}^{P,M}_{\eps,h}$, with $R=P=3$. For 
$$
\widehat{C}_{\eps,h}^M \, \grad \widehat{u}_h \in \left\{ C^{N,M}_{\eps,h} \, \grad u_{\star,h}, \ C^{N,M}_{\eps,h} \, \grad u_{\star,h}^{N,M}, \ \overline{C}^{R,M}_{\eps,h} \, \grad\overline{u}^{P,M}_{\eps,h} \right\},
$$ 
we define the criterion
\begin{equation} \label{eq:sta.h1}
\verb?err_sto_H1? = \left( \frac{ \dis \sup_{f\in V^{\cal Q}_{\rm n}({\cal D})} \norm[L^2({\cal D} \setminus{\cal B})]{\grad u_{\eps,h}^M(f) - \widehat{C}_{\eps,h}^M \ \grad \widehat{u}_h(f)}^2}{\norm[L^2({\cal D} \setminus{\cal B})]{\grad u_{\eps,h}^M \left( \widehat{f}_\eps \right)}^2}\right)^{1/2},
\end{equation}
where, here again, the supremum is taken over the space $V^{\cal Q}_{\rm n}({\cal D})$ for ${\cal Q}=16 \gg P$, $\widehat{f}_\eps\in V^{\cal Q}_{\rm n}({\cal D})$ denotes the argument of the $\sup$ problem, and boundary elements ${\cal B}$ are removed from the evaluation criterion, as in the periodic case~\eqref{eq:per.h1}. We recall that, in~\eqref{eq:sta.h1}, $u_{\eps,h}^M(f)$ is the empirical mean~\eqref{eq:expec} over $M$ realizations of $u_{\eps,h}(f;\omega)$. It is thus an approximation to $\mathbb{E} \left[ u_\eps(f) \right]$. 

\medskip

The numerical results are collected in Table~\ref{ta:3}. We see that our surrogate defines an approximation of $\mathbb{E}(\grad u_\eps)$ which is systematically better than that provided by the classical homogenization approach, for any choice of $h$ and $M$.

\medskip

\begin{verbbox}err_sto_H1\end{verbbox}
\begin{table}[h!]
\begin{center}
\begin{tabular}{|c||c|c|c|}
\hline
$N$ & $8$ & $16$ & $32$ \\
\hline
\verb?err_sto_H1? for $C^{N,M}_{\eps,h} \grad u_{\star,h}$ ($\eps/h\approx 27$, $M=100$) & $1.043 \ 10^{-1}$ & $9.635 \ 10^{-2}$ & $9.394 \ 10^{-2}$ \\
\hline
\hspace{4.5cm} ($\eps/h\approx 108$, $M=100$) & $8.648 \ 10^{-2}$ & $8.120 \ 10^{-2}$ & $8.010 \ 10^{-2}$ \\
\hline
\hspace{4.7cm} ($\eps/h\approx 27$, $M=400$) & $8.542 \ 10^{-2}$ & $7.828 \ 10^{-2}$ & $7.298 \ 10^{-2}$ \\
\hline
\hspace{4.7cm} ($\eps/h\approx 54$, $M=400$) & $6.599 \ 10^{-2}$ & $6.222 \ 10^{-2}$ & $6.067 \ 10^{-2}$ \\
\hline
\hline
\verb?err_sto_H1? for $C^{N,M}_{\eps,h} \grad u^{N,M}_{\star,h}$ ($\eps/h\approx 27$, $M=100$) & $9.799 \ 10^{-2}$ & $9.095 \ 10^{-2}$ & $8.961 \ 10^{-2}$ \\
\hline
\hspace{4.8cm} ($\eps/h\approx 108$, $M=100$) & $8.620 \ 10^{-2}$ & $8.022 \ 10^{-2}$ & $7.952 \ 10^{-2}$ \\
\hline
\hspace{5.cm} ($\eps/h\approx 27$, $M=400$) & $7.605 \ 10^{-2}$ & $7.173 \ 10^{-2}$ & $6.780 \ 10^{-2}$ \\
\hline
\hspace{5.cm} ($\eps/h\approx 54$, $M=400$) & $6.142 \ 10^{-2}$ & $5.957 \ 10^{-2}$ & $5.872 \ 10^{-2}$ \\
\hline
\hline
\verb?err_sto_H1? for $\overline{C}^{R,M}_{\eps,h} \grad \overline{u}^{P,M}_{\eps,h}$ ($\eps/h\approx 27$, $M=100$) & $6.000 \ 10^{-2}$ & $4.542 \ 10^{-2}$ & $3.018 \ 10^{-2}$ \\
\hline
\hspace{4.7cm} ($\eps/h\approx 108$, $M=100$) & $5.912 \ 10^{-2}$ & $4.657 \ 10^{-2}$ & $3.596 \ 10^{-2}$ \\
\hline
\hspace{5.cm} ($\eps/h\approx 27$, $M=400$) & $3.030 \ 10^{-2}$ & $3.814 \ 10^{-2}$ & $2.625 \ 10^{-2}$ \\
\hline
\hspace{5.cm} ($\eps/h\approx 54$, $M=400$) & $5.157 \ 10^{-2}$ & $3.613 \ 10^{-2}$ & $2.849 \ 10^{-2}$ \\
\hline
\end{tabular}
\end{center}
\caption{Approximation of $\mathbb{E}(\grad u_\eps)$ in the $L^2$ norm (\theverbbox) by $C^{N,M}_{\eps,h} \grad u_{\star,h}$, $C^{N,M}_{\eps,h} \grad u^{N,M}_{\star,h}$ and $\overline{C}^{R,M}_{\eps,h} \grad \overline{u}^{P,M}_{\eps,h}$ in function of $N$ (the various lines correspond to various values of $h$ and $M$).} \label{ta:3}
\end{table}

\subsection{Results in the case $\eps\geq\overline{\eps}$} \label{sse:eps.sup}

In the regime $\eps\geq\overline{\eps}$, we quantitatively investigate whether the best constant matrix provided by our approach allows for an accurate approximation of the exact solution, in the $L^2$ norm in the sense of the criteria~\eqref{eq:per.l2} or~\eqref{eq:sta.l2}, and in the $H^1$ norm in the sense of the criteria~\eqref{eq:per.h1} or~\eqref{eq:sta.h1}. 

We also consider below the criterion~\eqref{eq:per.mat}, only in the periodic setting. It is indeed interesting to quantify the threshold value of $\eps$ above which $\overline{A}_\eps$ is significantly different from $A_\star$ (let alone to understand the practical limitation of homogenization theory).

When considering large values of the parameter $\eps$, it is necessary to consider $P$ right-hand sides with $P$ larger than $d(d+1)/2=3$, as pointed out in Section~\ref{ssse:sup}. This value depends on $\eps$ and is denoted $P(\eps)$.

\subsubsection{Results in the periodic setting} \label{ssse:eps.sup.per}

We consider the set $\{ \eps_k \}_{0\leq k\leq 2}$ of parameters introduced in Section~\ref{ssse:eps.inf.per}. For $0\leq k\leq 2$, we have $\eps_k\geq\overline{\eps}$. We choose the number of right-hand sides as $P(\eps_0)=9$ and $P(\eps_1)=P(\eps_2)=5$ (we recall that $P(\eps_k)=3$ for $3\leq k\leq 6$). Considering less right-hand sides significantly alters the approximation results, while considering more right-hand sides does not significantly improve these results.

We consider the evaluation criteria~\eqref{eq:per.mat},~\eqref{eq:per.l2} and~\eqref{eq:per.h1}. We keep ${\cal Q}=16$ functions in the test-space $V_{\rm n}^{\cal Q}({\cal D})$. For the $H^1$-reconstruction, we choose the number of right-hand sides $R(\eps)$ such that $R(\eps_0)=R(\eps_1)=5$ and $R(\eps_2)=3$ (which satisfies $R(\eps) \leq P(\eps)$). The numerical results for the approximation of the homogenized matrix, the $L^2$-approximation and the $H^1$-approximation, are respectively collected in Table~\ref{ta:4}, Figure~\ref{fi:5} and Table~\ref{ta:5}.

\medskip

We observe on Table~\ref{ta:4} that the approximation of the homogenized matrix provided by our approach highly improves when decreasing $\eps$ from $\eps=0.4$ to $\eps=0.2$. For $\eps \geq 0.4$, the homogenized matrix does not correctly describe the medium.

\begin{verbbox}err_per_mat\end{verbbox}
\begin{table}[h!]
\begin{center}
\begin{tabular}{|c||c|c|c|}
\hline
$\eps$ & $0.4$ & $0.2$ & $0.1$ \\
\hline
\verb?err_per_mat? & $3.8420 \ 10^{-2}$ & $3.7056 \ 10^{-3}$ & $1.8623 \ 10^{-3}$ \\
\hline
\end{tabular}
\end{center}
\caption{Approximation of $A_\star$ (\theverbbox) in function of $\eps$ (here $\eps/h \approx 43$).} \label{ta:4}
\end{table}

\bigskip

Figure~\ref{fi:5} confirms this observation when it comes to the solution itself. We have seen that, for $\eps=0.4$, $A_\star$ and $\overline{A}_\eps$ are significantly different. The solutions $u_\star$ and $\overline{u}_\eps = \overline{u}(\overline{A}_\eps)$ are also significantly different, the latter being a much better $L^2$-approximation of $u_\eps$ than the former or the first-order two-scale expansion. For smaller values of $\eps$, we already observe the behavior we have described in Section~\ref{ssse:eps.inf.per}. Similar comments apply to the approximation of $\grad u_\eps$ (see Table~\ref{ta:5}).

\medskip

\begin{verbbox}err_per_L2\end{verbbox}
\begin{figure}[h!]
\centering
\begin{tikzpicture}[scale=1.]
\begin{loglogaxis} [
xtick={0.4,0.2,0.1},
xticklabels={$0.4$,$0.2$,$0.1$}
]
\addplot[thick,mark=x,red] table[x=epsilon,y=err_l2_star]{dat/eps_sup/det/err_l2_43.dat};

\addplot[thick,mark=x,brown] table[x=epsilon,y=err_l2_star_corr]{dat/eps_sup/det/err_l2_43.dat};

\addplot[thick,mark=x,black] table[x=epsilon,y=err_l2_bar]{dat/eps_sup/det/err_l2_43.dat};
\end{loglogaxis}
\end{tikzpicture}
\caption{Approximation of $u_\eps$ in the $L^2$ norm (\theverbbox) by $u_{\star,h}$ (red), $u_{\eps,h}^{1,\vec{\theta}}$ (brown) and $\overline{u}^P_{\eps,h}$ (black) in function of $\eps$ (here $\eps/h \approx 43$). These quantities are defined in Section~\ref{ssse:eps.inf.per}.} \label{fi:5}
\end{figure}
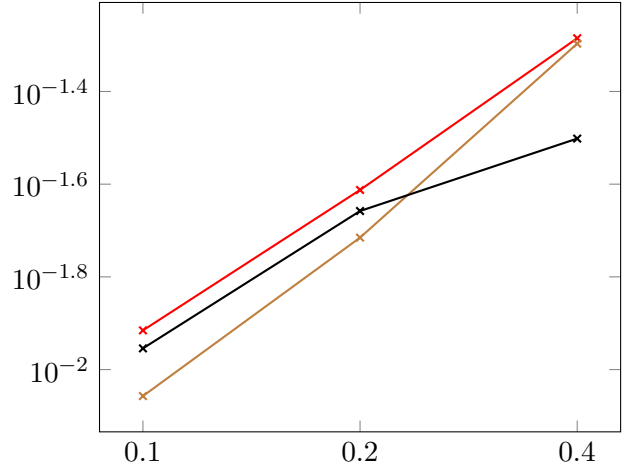

\medskip

\begin{verbbox}err_per_H1\end{verbbox}
\begin{table}[h!]
\begin{center}
\begin{tabular}{|c||c|c|c|}
\hline
$\eps$ & $0.4$ & $0.2$ & $0.1$ \\
\hline
\verb?err_per_H1? for $C_{\eps,h}\grad u_{\star,h}$ & $9.5890 \ 10^{-2}$ & $4.8421 \ 10^{-2}$ & $3.3923 \ 10^{-2}$ \\
\hline
\hline
\verb?err_per_H1? for $\overline{C}^R_{\eps,h}\grad\overline{u}^P_{\eps,h}$ & $8.7591 \ 10^{-2}$ & $5.8225 \ 10^{-2}$ & $3.2373 \ 10^{-2}$ \\
\hline
\end{tabular}
\end{center}
\caption{Approximation of $\grad u_\eps$ in the $L^2$ norm (\theverbbox) by $C_{\eps,h}\grad u_{\star,h}$ and $\overline{C}^R_{\eps,h}\grad\overline{u}^P_{\eps,h}$ in function of $\eps$ (here $\eps/h \approx 43$). See Section~\ref{ssse:eps.inf.per} for a definition of these quantities.} \label{ta:5}
\end{table}

\subsubsection{Results in the stationary ergodic setting} \label{ssse:eps.sup.sta}

We consider the sets $\{ \eps_k \}_{0\leq k\leq 2}$ and $\{ N_k \}_{0\leq k\leq 2}$ of parameters introduced in Section~\ref{ssse:eps.inf.sta}, for which we have $\eps_k>\overline{\eps}$. We choose the number of right-hand sides as $P(\eps_0)=9$ and $P(\eps_1)=P(\eps_2)=5$, and fix the number of Monte Carlo realizations to $M=100$.

We consider the evaluation criteria~\eqref{eq:sta.l2} and~\eqref{eq:sta.h1}, with ${\cal Q}=16$ functions in the test-space $V_{\rm n}^{\cal Q}({\cal D})$. For the $H^1$-reconstruction, the number of right-hand sides is chosen to be $R(\eps_0)=R(\eps_1)=5$ and $R(\eps_2)=3$. Note that again $R(\eps) \leq P(\eps)$. The numerical results for the $L^2$- and $H^1$-approximation are respectively collected in Figure~\ref{fi:6} and Table~\ref{ta:6}.

\begin{remark}
We note that, when working with $\eps = \eps_0 = 1/2$, we have, in view of~\eqref{eq:N}, $N=N_0=1$. In view of~\eqref{eq:sta.osc.1}--\eqref{eq:sta.osc.2}, it turns out that, in this case, there are only 16 different realizations of the field $a^{\rm sto}$. For this value of $\eps$, the expectation is computed by a simple enumeration of all the possible realizations. For $\eps = \eps_1 = 1/4$, there are already 65,536 realizations, and expectations are computed by empirical means over $M$ realizations.
\end{remark}

\medskip

On Figure~\ref{fi:6}, we observe that the solution associated with the best matrix we compute is an approximation of $\mathbb{E}(u_\eps)$ (in the $L^2$ norm) generally more accurate than the solution associated with the exact homogenized matrix (since here $N$ is small, the approximate matrix $A_\star^{N,M}$ is not expected to be an accurate approximation of $A_\star$). Table~\ref{ta:6} shows that our surrogate defines an approximation of $\mathbb{E}(\grad u_\eps)$, the accuracy of which is comparable, and often much better, to that provided by the homogenization approach. For the small values of $N$ considered here, our approach is less expensive than the classical homogenization approach.

\medskip

\begin{verbbox}err_sto_L2\end{verbbox}
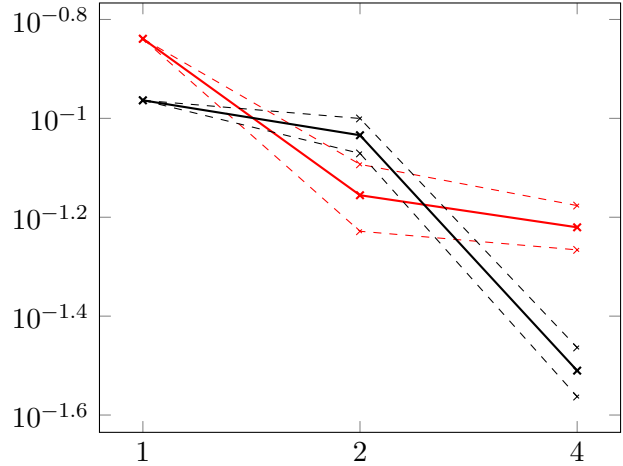
\begin{figure}[h!]
\centering
\begin{tikzpicture}[scale=1.]
\begin{loglogaxis} [
xtick={1,2,4},
xticklabels={$1$,$2$,$4$}
]
\addplot[thick,mark=x,red] table[x=N,y=err_l2_star]{dat/eps_sup/sto/err_l2_27_100_FRED_BAR.dat};
\addplot[dashed,mark=x,red] table[x=N,y=err_l2_star_M]{dat/eps_sup/sto/err_l2_27_100_FRED_BAR.dat};
\addplot[dashed,mark=x,red] table[x=N,y=err_l2_star_P]{dat/eps_sup/sto/err_l2_27_100_FRED_BAR.dat};

\addplot[thick,mark=x,black] table[x=N,y=err_l2_bar]{dat/eps_sup/sto/err_l2_27_100_FRED_BAR.dat};
\addplot[dashed,mark=x,black] table[x=N,y=err_l2_bar_M]{dat/eps_sup/sto/err_l2_27_100_FRED_BAR.dat};
\addplot[dashed,mark=x,black] table[x=N,y=err_l2_bar_P]{dat/eps_sup/sto/err_l2_27_100_FRED_BAR.dat};
\end{loglogaxis}
\end{tikzpicture}
\caption{Approximation of $\mathbb{E}(u_\eps)$ in the $L^2$ norm (\theverbbox) by $u_{\star,h}$ (red) and $\overline{u}_{\eps,h}^{P,M}$ (black) in function of $N$. For $N \geq 2$, all expectations are approximated by an empirical mean over $M=100$ realizations. Since $M$ is finite, results are random. We have performed the overall computation $10$ times and show the corresponding $95\%$ confidence interval (here $\eps/h \approx 27$).} \label{fi:6}
\end{figure}

\medskip

\begin{verbbox}err_sto_H1\end{verbbox}
\begin{table}[h!]
\begin{center}
\begin{tabular}{|c||c|c|c|}
\hline
$N$ & $1$ & $2$ & $4$ \\
\hline
\verb?err_sto_H1? for $C^{N,M}_{\eps,h} \grad u_{\star,h}$ & $1.4947 \ 10^{-1}$ & $1.3091 \ 10^{-1}$ & $1.0720 \ 10^{-1}$ \\
\hline
\verb?err_sto_H1? for $\overline{C}^{R,M}_{\eps,h} \grad \overline{u}^{P,M}_{\eps,h}$ & $1.0955 \ 10^{-1}$ & $1.4595 \ 10^{-1}$ & $6.9334 \ 10^{-2}$ \\
\hline
\end{tabular}
\end{center}
\caption{Approximation of $\mathbb{E}(\grad u_\eps)$ in the $L^2$ norm (\theverbbox) by $C^{N,M}_{\eps,h} \grad u_{\star,h}$ and $\overline{C}^{R,M}_{\eps,h} \grad \overline{u}^{P,M}_{\eps,h}$ in function of $N$, for $M=100$ and $\eps/h \approx 27$ (see Section~\ref{ssse:eps.inf.sta} for a definition of these quantities).} \label{ta:6}
\end{table}

\section*{Acknowledgments}

The authors would like to thank Albert Cohen (Universit\'e Pierre et Marie Curie) for stimulating and enlightning discussions about the work reported in this article, and in particular for suggesting the cost function in~\eqref{eq:infsup.0} in replacement of that in~\eqref{eq:infsup.old}, for providing the perspective of an optimization upon the class of matrices $\overline{A}$ that are considered, as detailed in Section~\ref{sec:outline}, and for carefully reading a preliminary version of this manuscript.

The authors also acknowledge several constructive comments by the two anonymous referees, which have allowed to improve (in particular with Remarks~\ref{re:darcy} and~\ref{rem:preuve_referee}) the original version of this manuscript.

The work of CLB, FL and SL is partially supported by EOARD under Grant FA8655-13-1-3061. The work of CLB and FL is also partially supported by ONR under Grants N00014-12-1-0383 and N00014-15-1-2777.

\appendix

\section{Proof of Proposition~\ref{pr:as.cons}} \label{se:proof}

\subsection{Preliminary results}

Before we are in position to show Proposition~\ref{pr:as.cons}, we first need to prove the following two preliminary lemmas, namely Lemma~\ref{le:test} and Lemma~\ref{le:recons}.

\begin{lemma} \label{le:test}
Under the assumptions~\eqref{eq:ass.per} and~\eqref{eq:ass.bound.per}, the following convergence holds:
\begin{equation} \label{eq:test}
\lim_{\eps\to 0} \Phi_\eps(A_\star) = 0.
\end{equation}
\end{lemma}
We recall that $\Phi_\eps$ is defined by~\eqref{eq:Phi.sup}: for any $\overline{A}$, 
$$
\Phi_\eps(\overline{A}) = \sup_{f\in\Ltwon} \Phi_\eps(\overline{A},f) = \sup_{f\in\Ltwon} \norm[\Ltwo]{ (-\Delta)^{-1} \left( \div(\overline{A}\grad u_\eps(f))+f\right)}^2.
$$

\begin{proof}[Proof of Lemma~\ref{le:test}]
We use the notations and results of Section~\ref{sse:infsup}. Let $f_\star^\eps\in L^2_{\rm n}({\cal D})$ such that
\begin{equation} \label{eq:proof1.1}
\Phi_\eps(A_\star) = \norm[\Ltwo]{(-\Delta)^{-1}\left(\div(A_\star\grad u_\eps(f_\star^\eps))+f_\star^\eps\right)}^2,
\end{equation}
and let $C_{\rm P}>0$ be a Poincar\'e constant for ${\cal D}$, namely a constant such that, for any $v \in H^1_0({\cal D})$, we have $\norm[\Ltwo]{v} \leq C_{\rm P} \norm[\Ltwo]{\grad v}$. 

Using standard a priori estimates, we have, for any $f \in \Ltwo$, that
\begin{equation}
\label{eq:maison1}
\norm[\Ltwo]{(-\Delta)^{-1} f } \leq C_{\rm P}^2 \, \norm[\Ltwo]{f}.
\end{equation}
Using that $\alpha \leq A_\eps \leq \beta$ (see~\eqref{eq:ass.bound.per}), we likewise get that, for any $f \in \Ltwo$, 
\begin{equation}
\label{eq:apriori}
\norm[\Ltwo]{\grad u_\eps(f)} \leq \frac{C_{\rm P}}{\alpha} \norm[\Ltwo]{f}.
\end{equation}
We now estimate $z_\eps = (-\Delta)^{-1}\left( \div(A_\star\grad u_\eps(f)) \right)$. We recall that~\eqref{eq:ass.bound.per} implies that
\begin{equation}
\label{eq:maison2}
\alpha \leq A_\star \leq \beta.
\end{equation}
From the variational formulation satisfied by $z_\eps$, we obtain $\norm[\Ltwo]{\grad z_\eps} \leq | A_\star | \, \norm[\Ltwo]{\grad u_\eps(f)}$, which implies, using~\eqref{eq:apriori} and~\eqref{eq:maison2}, that $\norm[\Ltwo]{\grad z_\eps} \leq C_{\rm P} \, \beta/\alpha \, \norm[\Ltwo]{f}$, hence
\begin{equation}
\label{eq:maison3}
\norm[\Ltwo]{(-\Delta)^{-1}\left( \div(A_\star\grad u_\eps(f)) \right)} \leq C_{\rm P}^2 \, \frac{\beta}{\alpha} \, \norm[\Ltwo]{f}.
\end{equation}
Using~\eqref{eq:proof1.1}, \eqref{eq:maison3}, \eqref{eq:maison1} and the fact that $\norm[\Ltwo]{f_\star^\eps}=1$ for all $\eps>0$, we deduce that the sequence $\left\{ \Phi_\eps(A_\star) \right\}_{\eps>0}$ is uniformly bounded. There thus exists a subsequence, that we still denote by $\left\{ \Phi_\eps(A_\star) \right\}_{\eps>0}$, that converges in $\mathbb{R}$. Let us denote by $\overline{\Phi}$ its limit. We prove in the sequel that $\overline{\Phi} = 0$, which implies~\eqref{eq:test}.

\medskip

Since $\left\{ f_\star^\eps \right\}_{\eps>0}$ is uniformly bounded in $\Ltwo$, there exists a subsequence, again denoted $\left\{ f_\star^\eps \right\}_{\eps>0}$, that weakly converges in $\Ltwo$ when $\eps\to 0$ to some function $f_\star^0\in L^2({\cal D})$ which satisfies $\norm[\Ltwo]{f_\star^0} \leq 1$. From~\eqref{eq:proof1.1}, we infer, by the triangle inequality,
\begin{equation} \label{eq:proof1.2}
\left( \Phi_\eps(A_\star) \right)^{1/2} \leq I^\eps_1 + I^\eps_2 + I^\eps_3,
\end{equation}
with
\begin{eqnarray*}
I^\eps_1
&=&
\norm[\Ltwo]{(-\Delta)^{-1} \left( \div( A_\star \grad u_\eps(f_\star^\eps - f_\star^0))\right)},
\\
I^\eps_2
&=&
\norm[\Ltwo]{(-\Delta)^{-1}\left(\div(A_\star\grad u_\eps(f_\star^0))+f_\star^0\right)},
\\
I^\eps_3
&=&
\norm[\Ltwo]{(-\Delta)^{-1}(f_\star^\eps-f_\star^0)}.
\end{eqnarray*}
We successively show that $I^\eps_1$, $I^\eps_2$ and $I^\eps_3$ vanish with $\eps$.

\medskip

\noindent
{\bf Step 1: estimation of $I_1^\eps$.} Let $z_\eps = (-\Delta)^{-1}\left(\div \left[ A_\star\grad (u_\eps(f_\star^\eps-f_\star^0)) \right] \right)\in H^1_0({\cal D})$. We have
$$
\norm[\Ltwo]{\grad z_\eps}^2
=
-\int_{\cal D}A_\star\grad (u_\eps(f_\star^\eps-f_\star^0))\cdot\grad z_\eps
\leq 
\beta\norm[\Ltwo]{\grad(u_\eps(f_\star^\eps-f_\star^0))}\norm[\Ltwo]{\grad z_\eps},
$$
where we have used~\eqref{eq:maison2}. Using the Poincar\'e inequality, we deduce
$$
I_1^\eps
=
\norm[\Ltwo]{z_\eps}
\leq
C_{\rm P} \, \beta \, \norm[\Ltwo]{\grad(u_\eps(f_\star^\eps-f_\star^0))},
$$
thus, using~\eqref{eq:osc.0}, we get that
\begin{equation} \label{eq:proof1.3}
\left( I_1^\eps \right)^2
\leq
C_{\rm P}^2 \, \frac{\beta^2}{\alpha} \, \int_{\cal D} A_\eps \grad(u_\eps(f_\star^\eps-f_\star^0))\cdot\grad(u_\eps(f_\star^\eps-f_\star^0))
=
C_{\rm P}^2 \, \frac{\beta^2}{\alpha} \, \int_{\cal D}(f_\star^\eps-f_\star^0)\;u_\eps(f_\star^\eps-f_\star^0).
\end{equation}
From~\eqref{eq:apriori}, we also deduce
$$
\norm[\Ltwo]{\grad(u_\eps(f_\star^\eps-f_\star^0))} 
\leq 
\frac{C_{\rm P}}{\alpha} \, \norm[\Ltwo]{f_\star^\eps-f_\star^0} 
\leq
2 \, \frac{C_{\rm P}}{\alpha}.
$$
Using the Poincar\'e inequality, we obtain that the sequence $\left\{ u_\eps(f_\star^\eps-f_\star^0) \right\}_{\eps>0}$ is uniformly bounded in $H^1({\cal D})$. There thus exists a subsequence, that we again denote $\left\{ u_\eps(f_\star^\eps-f_\star^0) \right\}_{\eps>0}$, which is strongly convergent in $\Ltwo$. The right-hand side of~\eqref{eq:proof1.3} is therefore the $L^2$ product of a sequence that weakly converges to 0 times a sequence that strongly converges. We hence deduce from~\eqref{eq:proof1.3} that
\begin{equation}
\label{eq:resu1}
\lim_{\eps \to 0} I^\eps_1 = 0.
\end{equation}

\medskip

\noindent
{\bf Step 2: estimation of $I_2^\eps$.} Let $w_\eps = \div(A_\star\grad u_\eps(f_\star^0))+f_\star^0$, $r_\eps = (-\Delta)^{-1} w_\eps \in H^1_0({\cal D})$ and $p_\eps = (-\Delta)^{-1} r_\eps \in H^1_0({\cal D})$. Using the definition of $p_\eps$, we have
\begin{equation}
\label{eq:inter1}
\left( I_2^\eps \right)^2 
= 
\int_{\cal D} r^2_\eps 
= 
\int_{\cal D} \grad r_\eps \cdot \grad p_\eps.
\end{equation}
Using the definition of $r_\eps$, we have, for any $\phi \in H^1_0({\cal D})$,
\begin{equation}
\label{eq:inter2}
\int_{\cal D} \grad r_\eps \cdot \grad \phi
=
- \int_{\cal D} A_\star \grad u_\eps(f_\star^0) \cdot \grad \phi + \int_{\cal D} f_\star^0 \ \phi.
\end{equation}
Using~\eqref{eq:inter2} for $\phi \equiv p_\eps$, \eqref{eq:inter1} reads as
\begin{equation}
\label{eq:inter3}
\left( I_2^\eps \right)^2 
= 
- \int_{\cal D} A_\star \grad u_\eps(f_\star^0) \cdot \grad p_\eps + \int_{\cal D} f_\star^0 \ p_\eps.
\end{equation}
In order to pass to the limit $\eps \to 0$ in~\eqref{eq:inter3}, we establish some bounds. Using~\eqref{eq:inter2} with $\phi \equiv r_\eps$ and the bounds~\eqref{eq:maison2}, we deduce
$$
\norm[\Ltwo]{\grad r_\eps} \leq \beta \norm[\Ltwo]{\grad u_\eps(f_\star^0)} + C_{\rm P} \norm[\Ltwo]{f_\star^0},
$$
which (together with the Poincar\'e inequality and~\eqref{eq:apriori}) implies that $r_\eps$ is uniformly bounded in $H^1({\cal D})$. There thus exists $r_0 \in H^1_0({\cal D})$ such that, up to some extraction, $r_\eps$ converges to $r_0$, weakly in $H^1({\cal D})$ and strongly in $L^2({\cal D})$.

Passing to the limit $\eps \to 0$ in~\eqref{eq:inter2}, and using that $\grad u_\eps(f)$ weakly converges to $\grad u_\star(f)$, we deduce that, for any $\phi \in H^1_0({\cal D})$,
$$
\int_{\cal D} \grad r_0 \cdot \grad \phi
=
- \int_{\cal D} A_\star \grad u_\star(f_\star^0) \cdot \grad \phi + \int_{\cal D} f_\star^0 \ \phi = 0,
$$
in view of the variational formulation of~\eqref{eq:hom.0}. We hence get that $r_0 \equiv 0$. 

We now turn to $p_\eps$. We have $p_\eps = (-\Delta)^{-1} r_\eps \in H^1_0({\cal D})$ and $r_\eps$ converges to $r_0 = 0$, weakly in $H^1({\cal D})$ and strongly in $L^2({\cal D})$. Hence $p_\eps$ converges to 0 strongly in $H^1_0({\cal D})$. 

We now pass to the limit $\eps \to 0$ in~\eqref{eq:inter3}, and obtain
\begin{equation}
\label{eq:resu2}
\lim_{\eps \to 0} I_2^\eps = 0.
\end{equation}

\medskip

\noindent
{\bf Step 3: estimation of $I_3^\eps$.} Let $k_\eps = (-\Delta)^{-1}(f_\star^\eps-f_\star^0)$. We have
\begin{equation}
\label{eq:ap3}
\norm[\Ltwo]{\grad k_\eps}^2 = \int_{\cal D} (f_\star^\eps-f_\star^0) k_\eps,
\end{equation}
hence, using the Poincar\'e inequality,
$$
\norm[\Ltwo]{\grad k_\eps} \leq C_{\rm P} \norm[\Ltwo]{f_\star^\eps-f_\star^0} \leq 2 \, C_{\rm P}.
$$
The sequence $\{ k_\eps \}_{\eps>0}$ is thus uniformly bounded in $H^1({\cal D})$ and there exists a subsequence, that we again denote $\{ k_\eps \}_{\eps>0}$, which is strongly convergent in $\Ltwo$. Using that $f_\star^\eps-f_\star^0$ weakly converges to 0 in $\Ltwo$, we deduce from~\eqref{eq:ap3} that $\dis \lim_{\eps \to 0} \norm[\Ltwo]{\grad k_\eps}^2 = 0$, thus, again using the Poincar\'e inequality,
\begin{equation}
\label{eq:resu3}
\lim_{\eps \to 0} I_3^\eps = \lim_{\eps \to 0} \norm[\Ltwo]{k_\eps} = 0.
\end{equation}

\medskip

\noindent
{\bf Conclusion.} Collecting~\eqref{eq:proof1.2}, \eqref{eq:resu1}, \eqref{eq:resu2} and~\eqref{eq:resu3}, we obtain that $\Phi_\eps(A_\star)$ converges to zero as $\eps \to 0$. We thus have shown that $\overline{\Phi}=0$. The limit being independent of the subsequence that we have considered, we eventually deduce that the whole sequence $\{ \Phi_\eps(A_\star) \}_{\eps>0}$ converges to zero. This completes the proof of Lemma~\ref{le:test}.
\end{proof}

\medskip

In what follows, we identify the set of indices $\left\{ (i,j), \ \ 1 \leq i \leq j \leq d \right\}$ with the set of indices $\dis \left\{ m, \ \ 1 \leq m \leq \frac{d(d+1)}{2} \right\}$.

\begin{lemma} \label{le:recons}
There exist $\dis \frac{d \, (d+1)}{2}$ functions $f_{\star,k}\in\Ltwon$ and $\dis \frac{d \, (d+1)}{2}$ functions $\varphi_{\star,k}\in C^\infty_0({\cal D})$ such that the matrix $Z_\star \in \R^{\frac{d(d+1)}{2} \times \frac{d(d+1)}{2}}$ defined by
\begin{equation} \label{eq:Z.mat}
\forall\, 1\leq k\leq\frac{d \, (d+1)}{2},\quad\forall\, 1\leq i<j\leq d, \quad\begin{cases}
\dis \left[Z_\star\right]_{k,(i,i)} = \int_{\cal D} u_{\star,k}\;\partial_{ii}\varphi_{\star,k},
\\ \noalign{\vskip 4pt}
\dis \left[Z_\star\right]_{k,(i,j)} = 2\int_{\cal D} u_{\star,k}\;\partial_{ij}\varphi_{\star,k},
\end{cases}
\end{equation}
where $u_{\star,k} = u_\star(f_{\star,k})$ is the solution to~\eqref{eq:hom.0} with right-hand side $f_{\star,k}$, is invertible.
\end{lemma}

\begin{proof}[Proof of Lemma~\ref{le:recons}]
In the Steps 1 and 2 below, we construct $f_{\star,k}\in L^2_{\rm n}({\cal D})$ and $\varphi_{\star,k}\in C^\infty_0({\cal D})$ inductively for $\dis 1 \leq k \leq d(d+1)/2$, such that the vector $\vec{E}_\star^k\in\mathbb{R}^{\frac{d(d+1)}{2}}$ defined by
\begin{equation} \label{eq:formula}
\forall\, 1\leq i<j\leq d,\qquad\begin{cases}
\dis \left[\vec{E}_\star^k\right]_{(i,i)} = \int_{\cal D}u_{\star,k}\;\partial_{ii}\varphi_{\star,k},
\\ \noalign{\vskip 3pt}
\dis \left[\vec{E}_\star^k\right]_{(i,j)} = 2\int_{\cal D}u_{\star,k}\;\partial_{ij}\varphi_{\star,k},
\end{cases}
\end{equation}
does not belong to $\text{Span}(\vec{E}_\star^1,\dots,\vec{E}_\star^{k-1})$. The vectors $\vec{E}_\star^1$, \dots, $\vec{E}_\star^{d(d+1)/2}$ being the rows of the matrix $Z_\star$, we deduce that $Z_\star$ is invertible.

\medskip

\noindent
\textbf{Step 1: Construction of $\vec{E}_\star^1$.} Choose $f_{\star,1} \in L^2_{\rm n}({\cal D})$ and $\varphi_{\star,1} \in C^\infty_0({\cal D})$ such that $\dis \int_{\cal D} f_{\star,1} \, \varphi_{\star,1}\neq 0$, and consider $\vec{E}^1_\star\in\mathbb{R}^{\frac{d(d+1)}{2}}$ defined by~\eqref{eq:formula} (where we recall that $u_{\star,1}$ is the solution to~\eqref{eq:hom.0} with right-hand side $f_{\star,1}$). Recalling that $A_\star$ is symmetric and constant, we have
$$
\sum_{1\leq i\leq j\leq d} \left[A_\star\right]_{i,j} \ \left[\vec{E}_\star^1\right]_{(i,j)}
=
-\int_{\cal D} A_\star \grad u_{\star,1} \cdot \grad \varphi_{\star,1}
=
-\int_{\cal D} f_{\star,1} \ \varphi_{\star,1} \neq 0,
$$
hence $\vec{E}_\star^1 \neq 0$.

\medskip

\noindent
\textbf{Step 2: Induction.} We assume that we have constructed $f_{\star,1}$, \dots, $f_{\star,k-1}$ and $\varphi_{\star,1}$, \dots, $\varphi_{\star,k-1}$ such that the family $\vec{E}_\star^1$, \dots, $\vec{E}_\star^{k-1}$ is free, for $\dis k \leq d(d+1)/2$. We now construct $f_{\star,k}\in L^2_{\rm n}({\cal D})$ and $\varphi_{\star,k}\in C^\infty_0({\cal D})$ such that the vector $\vec{E}_\star^k\in\mathbb{R}^{\frac{d(d+1)}{2}}$ defined in~\eqref{eq:formula} does not belong to $\text{Span}(\vec{E}_\star^1,\dots,\vec{E}_\star^{k-1})$. 

We proceed by contradiction and assume that, for any such $f_{\star,k}$ and $\varphi_{\star,k}$, there exist $\lambda_\ell(f_{\star,k},\varphi_{\star,k}) \in \mathbb{R}$, $1\leq \ell\leq k-1$, such that 
$$
\vec{E}_\star^k = \sum_{\ell=1}^{k-1} \lambda_\ell(f_{\star,k},\varphi_{\star,k}) \, \vec{E}_\star^\ell.
$$
For any vector $\vec{S}_\star\in\mathbb{R}^{\frac{d(d+1)}{2}}$, we have
$$
\sum_{1\leq i\leq j\leq d} \int_{\cal D} \left[ \widehat{\vec{S}}_\star \right]_{(i,j)} \; \partial_{ij} u_{\star,k} \; \varphi_{\star,k}
=
\sum_{1\leq i\leq j\leq d} [\vec{S}_\star]_{(i,j)} \left[\vec{E}_\star^k\right]_{(i,j)}
=
\sum_{\ell=1}^{k-1} \lambda_\ell(f_{\star,k},\varphi_{\star,k}) \, \vec{S}_\star \cdot \vec{E}_\star^\ell,
$$
where, for any $\vec{S} \in \mathbb{R}^{\frac{d(d+1)}{2}}$ and $\vec{E} \in \mathbb{R}^{\frac{d(d+1)}{2}}$, we denote $\dis \vec{S} \cdot \vec{E} = \sum_{m=1}^{d(d+1)/2} [\vec{S}]_m \, [\vec{E}]_m$, and where $\widehat{\vec{S}_\star} \in \mathbb{R}^{\frac{d(d+1)}{2}}$ is defined, for any $1 \leq i < j \leq d$, by
$$
\left[ \widehat{\vec{S}}_\star \right]_{(i,i)} = \left[ \vec{S}_\star \right]_{(i,i)},
\qquad
\left[ \widehat{\vec{S}}_\star \right]_{(i,j)} = 2 \left[ \vec{S}_\star \right]_{(i,j)}.
$$
Since $\dis k-1< d(d+1)/2$, there exists $\vec{S}_\star\in\mathbb{R}^{\frac{d(d+1)}{2}}$, $\vec{S}_\star\neq\vec{0}$, such that $\vec{S}_\star\cdot\vec{E}_\star^\ell=0$ for all $1\leq \ell\leq k-1$, and thus
$$
\forall \varphi_{\star,k} \in C^\infty_0({\cal D}), \qquad \sum_{1\leq i\leq j\leq d} \int_{\cal D} \left[ \widehat{\vec{S}}_\star \right]_{(i,j)} \; \partial_{ij} u_{\star,k} \; \varphi_{\star,k}=0.
$$
Since $\vec{S}_\star$ (and thus $\widehat{\vec{S}}_\star$) only depends on $\vec{E}_\star^1$, \dots, $\vec{E}_\star^{k-1}$ and not on $\varphi_{\star,k}$, this implies
$$
\sum_{1\leq i\leq j\leq d} \left[ \widehat{\vec{S}}_\star \right]_{(i,j)} \; \partial_{ij} u_{\star,k}=0 \ \ \text{in the sense of distributions,}
$$
thus 
$$
0=-\sum_{1\leq i\leq j\leq d} \left[ \widehat{\vec{S}}_\star \right]_{(i,j)} \; \partial_{ij} \div \left[ A_\star \grad u_{\star,k} \right] 
=
\sum_{1\leq i\leq j\leq d} \left[ \widehat{\vec{S}}_\star \right]_{(i,j)} \; \partial_{ij}f_{\star,k},
$$
for any $f_{\star,k}\in L^2_{\rm n}({\cal D})$. Since $\dis \left[ \widehat{\vec{S}}_\star \right]_{(i,j)}$ does not depend on $f_{\star,k}$, this shows that $\widehat{\vec{S}}_\star$, and thus $\vec{S}_\star$, vanishes. We reach a contradiction. We thus obtain the existence of $f_{\star,k}\in L^2_{\rm n}({\cal D})$ and $\varphi_{\star,k}\in C^\infty_0({\cal D})$ such that the vectors $\vec{E}_\star^1$, \dots, $\vec{E}_\star^{k-1}$, $\vec{E}_\star^k$ form a free family.
\end{proof}

\subsection{Proof of Proposition~\ref{pr:as.cons}}
\label{sec:app_preuve}

We can now perform the proof of Proposition~\ref{pr:as.cons}. The convergence~\eqref{eq:test} proved in Lemma~\ref{le:test} readily shows~\eqref{eq:lim.Ieps}. We are left with showing~\eqref{eq:lim.Abareps}. Using the functions $f_{\star,k} \in \Ltwon$ and $\varphi_{\star,k} \in C^\infty_0({\cal D})$ defined by Lemma~\ref{le:recons}, we introduce the matrix $Z_\eps\in\R^{\frac{d(d+1)}{2}\times\frac{d(d+1)}{2}}$ defined by
\begin{equation*} 
\forall\, 1\leq k\leq\frac{d(d+1)}{2},\quad \forall\, 1\leq i<j\leq d,\quad\begin{cases}
\dis \left[Z_\eps\right]_{k,(i,i)} = \int_{\cal D} u_{\eps,k}\;\partial_{ii}\varphi_{\star,k},
\\ \noalign{\vskip 4pt}
\dis \left[Z_\eps\right]_{k,(i,j)} = 2\int_{\cal D} u_{\eps,k}\;\partial_{ij}\varphi_{\star,k},
\end{cases}
\end{equation*}
where $u_{\eps,k} = u_\eps(f_{\star,k})$ is the solution to~\eqref{eq:osc.0} with right-hand side $f_{\star,k}$. Note that, for the second index of $Z_\eps$, we have again identified the sets $\left\{ (i,j), \ 1 \leq i \leq j \leq d \right\}$ and $\dis \left\{ m, \ 1 \leq m \leq \frac{d(d+1)}{2} \right\}$.

Since $u_{\eps,k}$ converges to $u_{\star,k}$ in $L^2({\cal D})$, the matrix $Z_\eps$ converges to the matrix $Z_\star$ defined by~\eqref{eq:Z.mat} when $\eps$ goes to zero. We have proved in Lemma~\ref{le:recons} that the matrix $Z_\star$ is invertible. This implies that the matrix $Z_\eps$ is invertible for $\eps$ sufficiently small, and that $Z_\eps^{-1}$ is bounded independently of $\eps$.

\medskip

We now introduce the vectors $\overline{\vec{V}}^\flat_\eps$ and $\vec{V}_\star$ in $\mathbb{R}^{\frac{d(d+1)}{2}}$ such that
$$
\forall\, 1\leq i\leq j\leq d,
\qquad
\left[\overline{\vec{V}}^\flat_\eps\right]_{(i,j)} = \left[\overline{A}^\flat_\eps\right]_{i,j},
\qquad
\left[\vec{V}_\star\right]_{(i,j)} = \left[A_\star\right]_{i,j},
$$
where we recall that $\overline{A}^\flat_\eps$ is a quasi-minimizing sequence of the functional~\eqref{eq:Phi.sup} (see~\eqref{eq:estim}). It can easily be seen that, for any $\overline{A} \in {\cal S}$, denoting $\overline{\vec{V}} \in \mathbb{R}^{\frac{d(d+1)}{2}}$ the vector such that $[\overline{\vec{V}}]_{(i,j)} = \overline{A}_{i,j}$ for any $1\leq i\leq j\leq d$, the following holds: for any $1\leq k \leq d(d+1)/2$,
\begin{equation} \label{eq:Z.bar}
\left[Z_\eps \ \overline{\vec{V}}\right]_k
=
\int_{\cal D}u_{\eps,k}\,\div(\overline{A}\grad\varphi_{\star,k})
=
\int_{\cal D} \div(\overline{A}\grad u_{\eps,k}) \ \varphi_{\star,k}
=
- \int_{\cal D} (-\Delta)^{-1}\left[\div(\overline{A}\grad u_{\eps,k})\right] \, \Delta\varphi_{\star,k},
\end{equation}
where $Z_\eps \ \overline{\vec{V}} \in \mathbb{R}^{\frac{d(d+1)}{2}}$ is the product of the matrix $Z_\eps\in\R^{\frac{d(d+1)}{2}\times\frac{d(d+1)}{2}}$ by the vector $\overline{\vec{V}} \in \mathbb{R}^{\frac{d(d+1)}{2}}$: for any $1\leq k \leq d(d+1)/2$, $\dis \left[Z_\eps \ \overline{\vec{V}}\right]_k = \sum_{1 \leq i \leq j \leq d} \left[Z_\eps\right]_{k,(i,j)} [\overline{\vec{V}}]_{(i,j)}$.

Now, for any $f\in L^2_{\rm n}({\cal D})$, we observe that
\begin{equation} \label{eq:rem}
\begin{alignedat}{1}
\norm[\Ltwo]{(-\Delta)^{-1} \left[ \div\left(\overline{A}^\flat_\eps\grad u_\eps(f) \right)-\div \Big(A_\star\grad u_\eps(f) \Big) \right]}^2
&\leq 2\left(\Phi_\eps(\overline{A}^\flat_\eps)+\Phi_\eps(A_\star)\right) \\
&\leq 2\left(I_\eps+\eps+\Phi_\eps(A_\star)\right) \\
&\leq 2\left(2\Phi_\eps(A_\star)+\eps\right).
\end{alignedat}
\end{equation}
Hence, applying this to $f \equiv f_{\star,k}$, and owing to Lemma~\ref{le:test},
$$
\norm[\Ltwo]{(-\Delta)^{-1}\left[ \div \left( \overline{A}^\flat_\eps\grad u_{\eps,k} \right) \right] - (-\Delta)^{-1} \Big[ \div \Big( A_\star\grad u_{\eps,k} \Big) \Big]}
$$
vanishes with $\eps$, for any $1\leq k \leq d(d+1)/2$.

We next deduce from~\eqref{eq:Z.bar} that $Z_\eps(\overline{\vec{V}}^\flat_\eps-\vec{V}_\star)$ vanishes as $\eps \to 0$. Since $Z_\eps$ is invertible when $\eps$ is sufficiently small (with $Z_\eps^{-1}$ bounded independently of $\eps$), we obtain that $\dis \lim_{\eps \to 0} \overline{\vec{V}}^\flat_\eps = \vec{V}_\star$, which is exactly the claimed convergence~\eqref{eq:lim.Abareps}. This concludes the proof of Proposition~\ref{pr:as.cons}.

\begin{remark} \label{re:proof}
Since the above proof uses~\eqref{eq:rem} precisely for the functions $f_{\star,k}$, $1\leq k\leq d(d+1)/2$ (and not for all functions $f\in L^2_{\rm n}({\cal D})$), we observe that, in the $\inf\max$ formulation introduced in Remark~\ref{re:infmax}, we have $\overline{A}^{{\rm max},\flat}_\eps\to A_\star$ when $\eps\to 0$.
\end{remark}

\begin{remark}
\label{rem:preuve_referee}
We recall that our approach consists in considering the problem~\eqref{eq:infsup}, that is
$$
I_\eps = \inf_{\overline{A}\in{\cal S}} \ \Phi_\eps(\overline{A}),
$$
where $\Phi_\eps$ is defined by~\eqref{eq:Phi.sup}: for any $\overline{A}$, 
$$
\Phi_\eps(\overline{A}) = \sup_{f\in\Ltwon} \Phi_\eps(\overline{A},f) = \sup_{f\in\Ltwon} \norm[\Ltwo]{ (-\Delta)^{-1} \left( \div(\overline{A}\grad u_\eps(f))+f\right)}^2.
$$
We show here that, when $\eps$ is sufficiently small, the minimum $I_\eps$ is attained.

\medskip

Consider indeed a minimizing sequence $\overline{A}_\eps^\eta$, satisfying, for any $\eta >0$,
\begin{equation}
\label{eq:quasi-min}
I_\eps \leq \Phi_\eps(\overline{A}_\eps^\eta) \leq I_\eps + \eta.
\end{equation}
Similarly to~\eqref{eq:rem}, we observe that, for any $f\in L^2_{\rm n}({\cal D})$, 
$$
\begin{alignedat}{1}
\norm[\Ltwo]{(-\Delta)^{-1} \left[ \div\left(\overline{A}^\eta_\eps\grad u_\eps(f) \right)-\div \Big(A_\star\grad u_\eps(f) \Big) \right]}^2
&\leq 2\left(\Phi_\eps(\overline{A}^\eta_\eps)+\Phi_\eps(A_\star)\right) \\
&\leq 2\left(I_\eps+\eta+\Phi_\eps(A_\star)\right) \\
&\leq 2\left(2\Phi_\eps(A_\star)+\eta\right).
\end{alignedat}
$$
Using~\eqref{eq:Z.bar}, we have
$$
\Big| Z_\eps \, \left(\overline{\vec{V}}^\eta_\eps - \vec{V}_\star \right) \Big|
\leq
C \sup_{f \in \Ltwon} \norm[\Ltwo]{(-\Delta)^{-1} \left[ \div\left(\overline{A}^\eta_\eps\grad u_\eps(f) \right)-\div \Big(A_\star\grad u_\eps(f) \Big) \right]}
$$
where $C$ is a constant independent of $\eps$ and $\eta$ and where the vector $\overline{\vec{V}}^\eta_\eps \in \mathbb{R}^{\frac{d(d+1)}{2}}$ is defined by $\dis \left[\overline{\vec{V}}^\eta_\eps\right]_{(i,j)} = \left[\overline{A}^\eta_\eps\right]_{i,j}$ for any $1\leq i\leq j\leq d$. When $\eps$ is sufficiently small, the matrix $Z_\eps$ is invertible with $Z_\eps^{-1}$ bounded independently of $\eps$. We thus deduce from the two above estimates that
$$
\Big| \overline{\vec{V}}^\eta_\eps - \vec{V}_\star \Big|^2 \leq C \left(\Phi_\eps(A_\star)+\eta\right)
$$
for some $C$ independent of $\eps$ and $\eta$. The vector $\overline{\vec{V}}^\eta_\eps$ (resp. $\vec{V}_\star$) is the representation (as a vector in $\mathbb{R}^{\frac{d(d+1)}{2}}$) of the symmetric matrix $\overline{A}_\eps^\eta \in \mathbb{R}^{d \times d}$ (resp. $A_\star$). We hence equivalently write that
$$
\Big| \overline{A}^\eta_\eps - A_\star \Big|^2 \leq C \left(\Phi_\eps(A_\star)+\eta\right).
$$
This shows that the sequence $\overline{A}^\eta_\eps$ is bounded independently of $\eta$. Up to the extraction of a subsequence (that we still denote $\eta$ for the sake of simplicity), it thus converges to some symmetric matrix $\overline{A}^0_\eps$ when $\eta \to 0$. Since $A_\star$ is positive definite and since $\dis \lim_{\eps \to 0} \Phi_\eps(A_\star) = 0$, we get that $\overline{A}_\eps^0$ is also positive-definite.

Passing to the limit $\eta \to 0$ in~\eqref{eq:quasi-min}, and temporarily assuming that $\Phi_\eps$ is continuous, we get that $I_\eps = \Phi_\eps(\overline{A}_\eps^0)$. This concludes the proof that the minimum $I_\eps$ is indeed attained when $\eps$ is sufficiently small.

\medskip

We are left with showing the continuity of $\overline{A} \mapsto \Phi_\eps(\overline{A})$. For any two matrices $\overline{A}_1$ and $\overline{A}_2$ and any $f \in \Ltwo$, we compute that
\begin{multline*}
\Phi_\eps(\overline{A}_1,f) - \Phi_\eps(\overline{A}_2,f)
=
\norm[\Ltwo]{ (-\Delta)^{-1} \left[ \div \left( (\overline{A}_1-\overline{A}_2) \grad u_\eps(f) \right) \right]}^2
\\
+ 2 \left\langle (-\Delta)^{-1} \left[ \div \left( (\overline{A}_1-\overline{A}_2) \grad u_\eps(f) \right) \right], (-\Delta)^{-1} \left[ \div \left( \overline{A}_2 \grad u_\eps(f) \right) + f \right] \right\rangle_{L^2({\cal D})},
\end{multline*}
hence
$$
\left| \Phi_\eps(\overline{A}_1,f) - \Phi_\eps(\overline{A}_2,f) \right|
\leq
C \, \left| \overline{A}_1-\overline{A}_2 \right|^2 \ \norm[\Ltwo]{f}^2
+
C \, \left| \overline{A}_1-\overline{A}_2 \right| \ \norm[\Ltwo]{f}^2,
$$
where $C$ is independent of $f$ and $\overline{A}_1$. Taking the supremum over $f \in \Ltwon$, we thus deduce that
$$
\left| \Phi_\eps(\overline{A}_1) - \Phi_\eps(\overline{A}_2) \right|
\leq
C \, \left| \overline{A}_1-\overline{A}_2 \right|^2
+
C \, \left| \overline{A}_1-\overline{A}_2 \right|,
$$
which implies that $\dis \lim_{\overline{A}_1 \to \overline{A}_2} \Phi_\eps(\overline{A}_1) = \Phi_\eps(\overline{A}_2)$, and thus the continuity of $\Phi_\eps$.
\end{remark}

\section{Details on the algorithm to solve the discrete problem~\eqref{eq:infsup.app}}
\label{sec:appB}

Let $\Phi^{P,M}_{\eps,h}(\overline{A},\vec{c})$ be given by~\eqref{eq:Phi.app}. Using the fact that $\Phi^{P,M}_{\eps,h}(\overline{A},\vec{c})$ is quadratic with respect to $\vec{c} \in \R^P$, one can easily observe that
\begin{equation*} 
\Phi^{P,M}_{\eps,h}(\overline{A},\vec{c})=\vec{c}^T\,G_{\eps,h}^M(\overline{A})\,\vec{c},
\end{equation*}
where $G_{\eps,h}^M(\overline{A})$ is the $P\times P$ matrix defined, for any $1\leq p,q\leq P$, by
\begin{multline} \label{eq:Gmat}
\left[G_{\eps,h}^M(\overline{A})\right]_{p,q} = \frac{1}{2} \sum_{1\leq i,j,k,l\leq d} \left[{\cal K}_{\eps,h}^M\right]_{i,j,k,l,p,q} \ \overline{A}_{i,j} \ \overline{A}_{k,l}
\\
-\sum_{1\leq i,j\leq d} \left( \left[\mathbb{K}_{\eps,h}^M \right]_{i,j,p,q} + \left[\mathbb{K}_{\eps,h}^M\right]_{i,j,q,p} \right) \overline{A}_{i,j} + \left[K_h\right]_{p,q},
\end{multline}
where ${\cal K}_{\eps,h}^M$, $\mathbb{K}_{\eps,h}^M$ and $K_h$ are defined by~\eqref{eq:Btens.app}, \eqref{eq:Bmat.app} and~\eqref{eq:bscal.app}, respectively. 

\medskip

Using the fact that $\Phi^{P,M}_{\eps,h}(\overline{A},\vec{c})$ is also quadratic with respect to $\overline{A}$, we have that
$$
\Phi^{P,M}_{\eps,h}(\overline{A},\vec{c}) = \frac{1}{2} \sum_{1\leq i,j,k,l\leq d} \left[\mathbb{B}_{\eps,h}^{P,M}(\vec{c})\right]_{i,j,k,l} \ \overline{A}_{i,j} \ \overline{A}_{k,l}-\sum_{1\leq i,j\leq d} \left[B_{\eps,h}^{P,M}(\vec{c})\right]_{i,j} \ \overline{A}_{i,j} + b_h^P(\vec{c}),
$$
where $\mathbb{B}_{\eps,h}^{P,M}(\vec{c})$ is the $d\times d\times d\times d$ fourth-order tensor defined by
\begin{equation} 
\label{eq:Ktens}
\left[\mathbb{B}_{\eps,h}^{P,M}(\vec{c})\right]_{i,j,k,l} 
= 
\sum_{1\leq p,q\leq P} \left[{\cal K}_{\eps,h}^M\right]_{i,j,k,l,p,q} \ c_p \ c_q,
\end{equation}
$B_{\eps,h}^{P,M}(\vec{c})$ is the $d\times d$ matrix defined by
\begin{equation} 
\label{eq:Kmat}
\left[B_{\eps,h}^{P,M}(\vec{c})\right]_{i,j} 
= 
\sum_{1\leq p,q\leq P} \left(\left[\mathbb{K}_{\eps,h}^M\right]_{i,j,p,q} + \left[\mathbb{K}_{\eps,h}^M \right]_{i,j,q,p} \right) c_p \ c_q,
\end{equation}
and 
$$
b_h^P(\vec{c}) = \sum_{1\leq p,q\leq P} \left[K_h\right]_{p,q} \ c_p \ c_q,
$$ 
where ${\cal K}_{\eps,h}^M$, $\mathbb{K}_{\eps,h}^M$ and $K_h$ are defined by~\eqref{eq:Btens.app}, \eqref{eq:Bmat.app} and~\eqref{eq:bscal.app}, respectively. We remark, in light of the expressions~\eqref{eq:Btens.app} and~\eqref{eq:Bmat.app}, that 
$$
\left[\mathbb{B}_{\eps,h}^{P,M}(\vec{c})\right]_{i,j,k,l} = \left[\mathbb{B}_{\eps,h}^{P,M}(\vec{c})\right]_{k,l,i,j},
\qquad
\left[\mathbb{B}_{\eps,h}^{P,M}(\vec{c})\right]_{i,j,k,l} = \left[\mathbb{B}_{\eps,h}^{P,M}(\vec{c})\right]_{j,i,k,l},
$$
and $\left[B_{\eps,h}^{P,M}(\vec{c})\right]_{i,j} = \left[B_{\eps,h}^{P,M}(\vec{c})\right]_{j,i}$. 

\footnotesize
\bibliographystyle{plain}
\bibliography{clbflsl_v2}

\end{document}